\documentclass{amsart}
\usepackage[english]{babel}
\usepackage{amsfonts}
\usepackage{amssymb}
\usepackage{hyperref}
\usepackage{tensor}

\usepackage{amsmath}
\usepackage{amsthm}
\usepackage{amssymb}

\usepackage{mathtools}
\usepackage{enumerate}
\usepackage{tikz}
\usepackage{tikz-cd}
\usepackage{varwidth}
\usepackage{pigpen}
\usepackage{mathabx}
\usepackage{graphicx}
\graphicspath{ {./images/} }
\usetikzlibrary{decorations.pathmorphing}
\usepackage{etoolbox}

\makeatletter
\newtheorem*{rep@theorem}{\rep@title}
\newcommand{\newreptheorem}[2]{%
\newenvironment{rep#1}[1]{%
 \def\rep@title{#2 \ref{##1}}%
 \begin{rep@theorem}}%
 {\end{rep@theorem}}}
\makeatother

\makeatletter
\def\l@subsection{\@tocline{2}{0pt}{2.5pc}{5pc}{}}
\makeatother

\usetikzlibrary{shapes,backgrounds, calc, angles, positioning}
\makeatother

\newtheorem{theorem}{Theorem}
\newreptheorem{theorem}{Theorem}
\newtheorem{lemma}[theorem]{Lemma}

\newtheorem{proposition}[theorem]{Proposition}

\theoremstyle{definition}

\newtheorem{remark}[theorem]{Remark}
\newtheorem{definition}[theorem]{Definition}
\newtheorem{example}[theorem]{Example}

\newtheorem*{question}{Question}

\numberwithin{theorem}{section}
\numberwithin{equation}{theorem}

\newcommand{\CD}{\mathcal{D}}

\newcommand{\CF}{\mathcal{F}}
\newcommand{\CG}{\mathcal{G}}
\newcommand{\CH}{\mathcal{H}}
\newcommand{\CI}{\mathcal{I}}
\newcommand{\CU}{\mathcal{U}}
\newcommand{\CK}{\mathcal{K}}

\newcommand{\CM}{\mathcal{M}}
\newcommand{\CN}{\mathcal{N}}
\newcommand{\CO}{\mathcal{O}}
\newcommand{\CP}{\mathcal{P}}
\newcommand{\CT}{\mathcal{T}}

\newcommand{\CV}{\mathcal{V}}
\newcommand{\CW}{\mathcal{W}}

\newcommand{\QQ}{\mathbb{Q}}
\newcommand{\RR}{\mathbb{R}}

\newcommand{\ZZ}{\mathbb{Z}}

\newcommand{\BC}{\mathbf{C}}
\newcommand{\BD}{\mathbf{D}}

\newcommand{\BF}{\mathbf{F}}

\renewcommand{\hat}{\widehat}
\newcommand{\SingLieGrpd}{\mathbf{SingLieGrpd}}

\newcommand{\Lie}{\mathbf{Lie}}
\newcommand{\LocLieGrpd}{\mathbf{LocLieGrpd}}

\newcommand{\LieGrpd}{\mathbf{LieGrpd}}
\newcommand{\LieAlg}{\mathbf{LieAlg}}
\newcommand{\Alg}{\mathbf{Alg}}
\newcommand{\SingLocLieGrpd}{\mathbf{SingLocLieGrpd}}

\newcommand{\Diffgl}{\mathbf{Diffgl}}

\newcommand{\Euc}{\mathbf{Eucl}}
\newcommand{\Set}{\mathbf{Set}}

\newcommand{\QUED}{\mathbf{QUED}}
\newcommand{\Man}{\mathbf{Man}}

\newcommand{\s}{s}
\renewcommand{\t}{t}
\newcommand{\m}{m}
\renewcommand{\u}{u}
\renewcommand{\i}{i}

\renewcommand{\subset}{\subseteq}

\newcommand{\til}{\widetilde}
\newcommand{\dif}[1]{\ifstrequal{#1}{}{\mathop{d\,\!}}{\mathop{d#1}}}

\DeclareMathOperator{\pr}{pr}
\DeclareMathOperator{\Id}{Id}

\newcommand{\into}{\hookrightarrow}
\newcommand{\grpd}{\rightrightarrows}

\newcommand{\inv}{^{-1}}

\newcommand{\divi}{\mathbf{\delta}}
\newcommand{\fiber}[2]{\!\tensor*[^{}_{#1}]{\times}{^{}_{#2}}}

\newcommand{\comment}[1]{}
\newcommand{\darrow}{\arrow[d, shift left]\arrow[d, shift right]}

\title{On the integrability of Lie algebroids by diffeological spaces}
\author{Joel Villatoro}
\date{Sep 2023}
\address{Washington University in St Louis, Department of Mathematics}
\email{joelv@wustl.edu}

\begin{document}

\begin{abstract}
Lie's third theorem does not hold for Lie groupoids and Lie algebroids. In this article, we show that Lie's third theorem is valid within a specific class of diffeological groupoids that we call `singular Lie groupoids.' To achieve this, we introduce a subcategory of diffeological spaces which we call `quasi-etale.' Singular Lie groupoids are precisely the groupoid objects within this category, where the unit space is a manifold.

Our approach involves the construction of a functor that maps singular Lie groupoids to Lie algebroids, extending the classical functor from Lie groupoids to Lie algebroids. We prove that the Ševera-Weinstein groupoid of an algebroid is an example of a singular Lie groupoid, thereby establishing Lie's third theorem in this context.
\end{abstract}
\maketitle

\setcounter{tocdepth}{2}
\newpage
\tableofcontents
\newpage

\section{Introduction}
Lie theory provides us with a differentiation procedure that takes global symmetries geometric objects (Lie groupoids) as input and outputs infinitesimal symmetry algebras (Lie algebroids). This differentiation procedure takes the form of a functor:
\[\Lie \colon \LieGrpd \to \LieAlg \]
Given infinitesimal data, such as a morphism or object in \( \LieAlg\), the ``integration problem'' refers to constructing a corresponding global data in \( \LieGrpd \).

In the classical case of finite dimensional Lie groups and Lie algebras, Lie's second theorem Lie's second and third theorems are concerned with the existence of integrations of morphisms and objects, respectively. If one replaces the classical condition of ``simply connected'' with ``source simply-connected'', Lie's second theorem holds for the more general setting of Lie groupoids. On the other hand, Lie's third theorem is known to be false. In other words, there exist Lie algebroids which are not integrated by any Lie groupoid. The first example of a non-integrable algebroid is due to Rui Almeida and Pierre Molino~\cite{almeida_suites_1985}.

To any algebroid \( A \) it is possible to functorially associate a kind of fundamental groupoid \( \Pi_1(A) \) which is commonly called the ``Weinstein groupoid'' or occasionally the ``Ševera-Weinstein'' groupoid.
It turns out that a Lie algebroid admits an integration if and only if \( \Pi_1(A) \) is a smooth manifold.
Furthermore, if \( \Pi_1(A) \) is smooth it is the ``universal'' source simply-connected integration.

The earliest version of this fundamental groupoid construction is by Cattaneo and Felder\cite{cattaneo_poisson_2001} where they describe how to construct a symplectic groupoid integrating a Poisson manifold via Hamiltonian reduction on the space of cotangent paths. 
In 2000, a few months after the article of Cattaneo and Felder appeared on the arXiv, Pavol Ševera gave a talk where he proposed a version of this construction where the Hamiltonian action is reinterpreted as homotopies in the category of algebroids~\cite{severa_title_2005}. An advantage of this perspective was that it makes it clear that Cattaneo and Felder’s approach could be easily generalized to work for any Lie algebroid.
Marius Crainic and Rui Loja Fernandes~\cite{crainic_integrability_2003} were able to study the geometry of the space of algebroid paths in detail and consequently find a precise criteria the smoothness of the Ševera-Weinstein groupoid.
Crainic and Fernandes credit Alan Weinstein for suggesting a path-based approach to them and so they introduced the term ``Weinstein groupoid’'.

Even when the Ševera-Weinstein groupoid is not smooth, it is clear that it is far from being an arbitrary topological groupoid. For example, in the work of Hsian-Hua Tseng and Chenchang Zhu~\cite{tseng_integrating_2006} it is observed that it is the topological coarse moduli-space of a groupoid object in \'etale geometric stacks over manifolds. In other words, it is the orbit space of an \'etale Lie groupoid equipped with a ``stacky'' product.

\subsection{Main question}
Our aim in this article will be to provide a foundation for a version of Lie theory that includes the kinds of groupoids associated to non-integrable algebroids. 
We intend to do this in a way that will permit us to extend notions of Morita equivalence, symplectic groupoids and differentiation to this larger context while preserving Lie's second theorem.

Therefore, we aim to answer the following question:
\begin{question}
Does there exist a category \( \BC \) with the following properties: 
\begin{itemize}
    \item The category of smooth manifolds is a full subcategory of \( \BC \).
    \item If \( \SingLieGrpd \) is the category of groupoid objects \( \CG \grpd M \) in \( \BC \) where \( M \) is a smooth manifold, there exists a functor:
    \[ \hat \Lie \colon \SingLieGrpd \to \LieAlg  \]
    \item \( \LieGrpd \) is a full subcategory of   \( \SingLieGrpd \) and we have that: 
    \[ \hat \Lie|_{\LieGrpd} = \Lie  \]
    \item There is a notion of ``simply connected object'' in \( \BC \) which corresponds to being simply connected on the subcategory of manifolds. And, using this notion of simply connectedness, Lie's second and third theorems hold for \( \hat \Lie \).
\end{itemize}
\end{question}
Answering this question is not so straightforward. 
If we take \( \BC \) to be the (2,1)-category of étale geometric stacks we get all but the last bullet point~\cite{tseng_integrating_2005}\cite{zhu_lie_2008}. 
Another natural choice could be to take \( \BC \) to be the category of diffeological spaces. However, this category is so large that there appears to be little hope of defining a Lie functor in this context.

Another complication with answering this question is that the notion of groupoid object referenced in the statement of the question is somewhat subtle. 
For example, a Lie groupoid is not just a groupoid object in manifolds. It is a groupoid object in manifolds where the source (or equivalently target) map is a submersion. 
Therefore, when proposing such a category \( \BC\) we also need to propose a notion of ``submersion'' to go along with it.

The problem, as stated above, is a sort of ``Goldilocks problem.'' 
If we choose the category \( \BC \) to be too large, we have little hope of defining the Lie functor. 
If we choose \( \BC \) to be to small, it may be the case that not every Lie algebroid can be integrated by a groupoid object in \( \BC \).

\subsection{Our solution}

In this article, we will give a partial answer to this question. 
We introduce the category of ``quasi-étale diffeological spaces'' (\( \QUED \) for short) and submit that it is a solution to the problem stated above.
The category \( \QUED \) is, in many ways, very natural and generalizes existing notions such as orbifolds, quasifolds and similar types of diffeological structures.

We will show that the category \( \QUED \) indeed satisfies the first three bullet points of our question above. In order to do this we will also need to explain what a ``submersion'' is in this context, and we will construct the Lie functor \( \Lie \colon \SingLieGrpd \to \LieAlg \).

The main results of our paper can be summarized in the following two theorems:
\begin{theorem}\label{theorem:main.lie.functor}
Let \( \SingLieGrpd \) be the category of \( \QUED\)-groupoids where the space of objects is a smooth manifold. There exists a functor:
\[ \hat \Lie \colon \SingLieGrpd \to \LieAlg  \]
with the property that \( \hat \Lie |_{\LieGrpd} = \Lie \).
\end{theorem}
\begin{theorem}\label{theorem:main.wein.groupoid}
    Given a Lie algebroid, \( A \to M \), let \( \Pi_1(A) \grpd M  \) be the Ševera-Weinstein groupid of \( A \) and consider it as a diffeological groupoid. Then \( \Pi_1(A) \) is an element of \( \SingLieGrpd \) and \( \hat\Lie(\Pi_1(A)) \) is canonically isomorphic to \( A \).
\end{theorem}
What does not appear in this article is a proof of Lie's second theorem and the notion of ``simply connected'' for \( \QUED \).
There is a diffeological version of simply-connected due to Patrick Iglesias-Zemmour~\cite{iglesias-zemmour_fibrations_1985}. Indeed, it is not too difficult to show that the source fibers of the Ševera-Weinstein groupoid are diffeologically simply-connected. However, a full proof of Lie's second theorem will require additional technical development which we intend to address in a later article.
\subsection{Methods and outline}
The main technical tools used in this article are the theory of diffeological spaces and structures as well as the the theory of local groupoids. The bulk of the article is dedicated to developing the technology needed to show that the Lie functor is well-defined. However, the basic idea behind our definition of the Lie functor for \( \SingLieGrpd\) is not fundamentally so complex: If \( \CG \grpd M \) is an element of \( \SingLieGrpd \) and \( \pi \colon \til \CG \to \CG \) is a ``local covering'' from a \emph{local Lie} groupoid \( \til \CG \) to \( \CG \) then we define the Lie algebroid of \( \CG \) to be the Lie algebroid of \( \til \CG \). In order for this definition to make sense, we will have to prove that every element of \( \SingLieGrpd \) is locally the quotient of a local Lie groupoid and this is the main technical difficulty.

In Section~\ref{section:diffeology}, we will review the basics of diffeology and diffeological spaces. We primarily do this to establish some notational conventions and to keep this article mostly self-contained. All of the material in this section is fairly standard in the field of diffeology and can be found in, for instance, the book of Patrick Iglesias-Zemmour~\cite{iglesias-zemmour_diffeology_2013}.

In Section~\ref{section:quasi-étale.diffeological.spaces}, we will introduce the notion of a quasi-étale diffeological space. Briefly, a quasi-étale diffeological space is a diffeological space which is locally the quotient of a smooth manifold by a ``nice'' equivalence relation. Nice in this context means that the equivalence classes are totally disconnected (e.g. think \( \QQ \subset \RR \) ) and the equivalence relation is ``rigid'' in the sense than any endomorphism of the equivalence relation must be an isomorphism.
After defining quasi-étale spaces we will explore a few properties of quasi-étale maps and show how it is possible to study maps between quasi-étale spaces in terms of maps between smooth manifolds by ``representing'' them on local charts. The remaining part of the section is dedicated to defining what a submersion in \( \QUED \) is and proving some key technical properties that we will need later.

In Section~\ref{section:quasi.étale.groupoids} we will define what a groupoid object in \( \QUED \) is as well as state the definition of a \emph{local} groupoid which is key to our differentiation procedure.

Section~\ref{section:differentiation} defines the Lie functor and states the main technical theorems that we need in order to prove that it is well-defined. It concludes with a classification of source connected singular Lie groupoid with integrable algebroids.

Section~\ref{section:proof.of.main.theorem} is dedicated to proving the main results that we rely on in Section~\ref{section:differentiation}. It is the most technical part of the article. The main tools used here are a combination of diffeology and the theory of local Lie groupoids. The main point of this section is to show that every element of \( \SingLieGrpd\) is locally the quotient of a local Lie groupoid in a ``unique'' way. With these proofs we finish the proof of Theorem~\ref{theorem:main.lie.functor}.

Finally, in Section~\ref{section:Weinstein.groupoid} we conclude by proving that the classical Ševera-Weinstein groupoid is an element of \( \SingLieGrpd \) and we explain how to apply the Lie functor to it. With this calculation we complete the proof of Theorem~\ref{theorem:main.wein.groupoid}.

\subsection{Related work}
The most closely related work to the subject of this article is perhaps that of Chenchang Zhu et al.(\cite{tseng_integrating_2006}\cite{zhu_lie_2008}\cite{bursztyn_principal_2020}). In Tseng and Zhu~\cite{tseng_integrating_2006} they describe the geometric stack whose orbit space is the Ševera-Weinstein groupoid. They also describe a differentiation procedure for certain groupoid objects in geometric stacks.

However, there are two main disadvantages to this approach: One is that it necessitates going from the category of Lie groupoids to a (2,1)-category of stacky groupoids. 
The inherently higher categorical nature comes with a variety of coherence conditions and other phenomena that can make working with such objects a bit cumbersome. 
The second problem, is with Lie's second theorem. Although stacky groupoids can be used to repair the loss of Lie's third theorem, this comes at the cost of further complicating Lie's second theorem. Unlike for ordinary Lie groupoids, for stacky groupoids, it is not the case that source simply connected groupoids operate as ``universal'' integrations. 
The version of Lie's second theorem that is true for stacky groupoids can be seen in \cite{zhu_lie_2008}. The stacky groupoids version replaces the ``source simply connected'' condition with the much stronger condition of ``source 2-connected''. Indeed, the simply-connected version of Lie's second theorem for stacky groupoids appears to be false.

Several other authors have also investigated the relationship between diffeology and Lie theory. For example,
Gilbert Hector and Enrique Macías-Virgós\cite{hector_diffeological_2002} wrote an article discussing diffeological groups with a particular emphasis on diffeomorphism groups. They use a diffeological version of the tangent functor to define a kind of Lie algebra that one can associate to a diffeological group. However, it is not clear from their work whether the resulting structure is indeed a Lie algebra in the classical sense. They do, however, show that this procedure recovers the expected Lie algebra in the case of diffeomorphism groups.

Also on the topic of diffeological groups, there is an article by Jean-Marie Souriau~\cite{Souriau_groupes_2006} where he generalizes a variety of properties of Lie groups to the diffeological setting. Some of them relate to the integration and differentiation problem. For example, Souriau observes that under some separability assumptions, it is possible to construct something akin to the exponential map.

There is work by Marco Zambon and Iakovos Androulidakis~\cite{androulidakis_integration_2020} on the topic of integrating singular subalgebroids by diffeological groupoids. In their work they develop a differentiation procedure for diffeological groupoids that arise as a kind of ``singular subgroupoid'' of an ambient Lie groupoid. Their integration and differentiation procedure is defined relative to an ambient (integrable) structure algebroid and they do not treat the case of non-integrable algebroids.

There is also work by Christian Blohmann~\cite{blohmann_elastic_2023} towards developing a kind of Lie theory/Cartan calculus for elastic diffeological spaces. The motivation behind this is to study the infinite dimensional symmetries that occur in field theory. However, we remark that the Ševera-Weinstein groupoid of a non-integrable algebroid does not appear to be elastic in general so it is not clear if such a theory would be suitable for the study of non-integrable algebroids.

\section*{Acknowledgements}
The author would like to thank Marco Zambon for the numerous discussions on this topic over the years, as well as his suggestions for this manuscript. The author has also greatly benefited from discussions with Christian Blohmann during his stay at the Max Planck Institute. We also acknowledge Cattaneo Felder, Rui Loja Fernandes, Eckhard Meinrenken, David Miyamoto, and Jordan Watts for their corrections and/or comments on an earlier draft of this article. This article is based on work that was supported by the following sources: Fonds Wetenschappelijk Onderzoek (FWO Project G083118N); the Max Planck Institute for Mathematics in Bonn; the National Science Foundation (Award Number 2137999).

\section*{Notation}
The category of sets will be denoted \( \Set \) and the category of finite-dimensional, Hausdorff, second countable, smooth manifolds will be written \( \Man \).
\section{Diffeology}\label{section:diffeology}
Diffeological spaces are a generalization of the notion of a smooth manifold. 
They were independently introduced by Souriau\cite{souriau_groupes_1984} and Chen\cite{chen_iterated_1977}. 
Fundamentally, the idea is to endow a space with a `manifold-like' structure by specifying which maps into the space are smooth. 
The standard textbook for the theory is by Iglesias-Zemmour\cite{iglesias-zemmour_diffeology_2013} who is also responsible for fleshing out a considerable amount of the standard diffeological techniques.
\subsection{Diffeological structures}
The core observation behind diffeology is that the smooth structure on a manifold \( M \) is completely determined by the set of smooth maps into \( M \) where the domains are Euclidean sets. 
The definition of diffeology is obtained by axiomatizing some of the basic properties of this distinguished collection of maps.
\begin{definition}
An \emph{\(n\)-dimensional Euclidean set} is an open subset \( U \subset \RR^n \). 
Let \( \Euc \) denote the category whose objects are Euclidean sets and whose morphisms are smooth functions between Euclidean sets.
\end{definition}
\begin{definition}
A \emph{diffeological structure} on a set \( X \) is a function \( \CD_X \) which assigns to each object \( U \in \Euc \) a distinguished collection \( \CD_X(U) \subset \Set(U,X) \). 
An element of \( \CD_X(U) \) for some \( U \) is call a \emph{plot}. 
Plots are required to satisfy the following axioms:
\begin{enumerate}[(D1)]
\item If \( \phi \colon U_\phi \to X \) is constant then \( \phi \) is a plot. 
\item If \( \phi \colon U_\phi \to X \) is a plot and \( \psi \colon V \to U_\phi \) is a morphism in \( \Euc \) then \( \phi \circ \psi \) is a plot. 
\item If \( \phi \colon U_\phi \to X \) is a function and \( \{ U_i \}  \) is an open cover of \( U_\phi \) such that \( \phi|_{U_i} \) is a plot for all \( i \in I \) then \( \phi \) is a plot.
\end{enumerate}
A \emph{diffeological space} is a pair \( (X, \CD_X) \) where \( X \) is a set and \( \CD_X \) is a diffeological structure on \( X \). 
In a mild abuse of notation we will typically refer to diffeological spaces by only their underlying set. 
However, it is important to keep in mind that a given set may admit many diffeological structures. 
Finally, given a plot \( \phi \) on a diffeological space \( X \) we will typically denote the domain of \( \phi \) using a subscript \( U_\phi \). 
\end{definition}
Diffeological spaces are quite general and include objects which range from the ordinary to quite pathological. Let us go over a few basic examples.
\begin{example}\label{example:diffeology.manifold}
If \( M \) is a smooth manifold and \( U \) is a Euclidean set, then we declare \( \phi \colon U_\phi \to M \) to be a plot if and only if \( \phi \) is smooth as a map of manifolds. 
\end{example}
\begin{example}\label{example:diffeology.topologicalspace}
Suppose \( X \) is a topological space. 
We can make \( X \) into a diffeological space by declaring any continuous function \( \phi \colon U \to X \) to be a plot. 
\end{example}
\begin{example}\label{example:diffeology.discrete}
Suppose \( X \) is a set. 
The \emph{discrete diffeology} on \( X \) is the unique diffeology on \( X \) for which every plot is locally constant. 
\end{example}
\begin{example}\label{example.diffeology.set}
Suppose \( X \) is a set. 
The \emph{coarse diffeology} on \( X \) is the diffeology on \( X \) which declares every set theoretic function \( \phi \colon U \to X \) from a Euclidean set to be a plot. 
\end{example}
\begin{example}\label{example:diffeology.subset}
Suppose \( X \) is a diffeological space and \( \iota \colon  Y  \to   X \) is the inclusion of an arbitrary subset. 
The \emph{subset diffeology} on \( Y \) is the diffeology on \( Y \) which says that \( \phi \colon U \to Y \) is a plot if and only if \( \iota \circ \phi \) is a plot. 
\end{example}
One particular example of the subset diffeology will be of particular relevance to this article.
\begin{definition}\label{definition:totally.disconnected}
    Suppose \( X \) is a diffeological space and \( Y \subset X\). We say that \( Y \) is \emph{totally disconnected} if the subset diffeology on \( Y \) is the discrete diffeology. In other words, a map \( \phi \colon U_\phi \to Y \) is a plot on \( X \) if and only if it is locally constant.
\end{definition}
\begin{example}
    The set of rational numbers \( \QQ \subset \RR \) is a totally disconnected subset of \( \RR \).
\end{example}

Diffeological spaces come with a natural topological structure. However, the relationship between topological structures and diffeological structures is much weaker than the usual one between smooth structures and topology.
\begin{definition}\label{definition:D.topology}
    Suppose \( X \) is a diffeological space. A subset \( V \subset X \) is said to be \emph{open} if for all plots \( \phi \colon U_\phi \to X \) we have that the inverse image \( \phi\inv(V) \subset U_\phi \) is open.
    This topology on \( X \) is called the \emph{D-topology}.
\end{definition}
From now on, whenever we refer to something being ``local'' or ``open'' in a diffeological space, we mean relative to the D-topology.
\subsection{Smooth maps}\label{subsection:smooth.maps}
Morphisms of diffeological spaces are defined in a rather straightforward way. Basically, 
a function is smooth if it pushes forward plots to plots. More formally:
\begin{definition}\label{definition:smooth.map}
Let \( X \) and \(Y \) be diffeological spaces. 
A function \( f \colon X \to Y \) is a \emph{smooth map} if for all plots \( \phi \) on \( X \) we have that \( f \circ \phi \) is a plot on \( Y \). 
We say that \( f \) is a \emph{diffeomorphism} if it is a bijection and the inverse function \( f^{-1} \) is smooth.

The category of diffeological spaces with smooth maps will be denoted \( \Diffgl \). 
\end{definition}
This notion of a smooth map extends the usual one for smooth maps between manifolds. Let us consider some basic examples.
\begin{example}\label{example:smooth.map.manifold}
Suppose \( M \) and \( N \) are smooth manifolds. 
A function \( f \colon M \to N \) is smooth as a map of diffeological spaces if and only if it is smooth as a map of manifolds. 
This means that the category of smooth manifolds embeds fully faithfully into the category of diffeological spaces. 
\end{example}
\begin{example}\label{example:smooth.map.subset}
Suppose \( X \) is a diffeological space and \( Y \) is a subset of \( X \) equipped with the subset diffeology. 
The inclusion map \( \iota \colon Y \to X \) is smooth. 
\end{example}
There are a few different notions of ``quotients'' in the context of diffeology. The most fundamental one is called a subduction. One can think of subduction as the diffeological version of a topological quotient.
\begin{definition}
Suppose \( X \) and \( Y \) are diffeological spaces. 
A smooth function \( f \) is called a \emph{subduction} 
if for all plots \( \phi \colon U_\phi \to Y \) and points \( u \in U \) 
there exists an open \( V \subset U_\phi \) of \( u \in U \) together with a plot \( \til \phi \colon V \to X \) 
such that \( \phi |_{V} = f \circ \til \phi \). 
\[
\begin{tikzcd}
 & X \arrow[d, "f"] \\
(U_\phi, u) \arrow[r, "\phi"]  & Y    
\end{tikzcd}
\quad \Rightarrow \quad
\begin{tikzcd}
 \exists (V, u) \arrow[d, "\iota", hook, dashed] \arrow[r, " \exists\til \phi", dashed] & X \arrow[d, "f"] \\
(U_\phi, u) \arrow[r, "\phi"]  & Y
\end{tikzcd}
\]
\end{definition}
Let us state a few examples.
\begin{example}\label{example:subduction.manifold}
If \( f \colon M \to N \) is a smooth map of manifolds, then \( f \) is a subduction if and only if for all \( p \in N \) there exists an element \( q \in M\) such that the differential \( T_q f \colon T_q M \to T_p N \) is a surjection.
\end{example}
\begin{example}\label{example:subduction.nebula}
Suppose \( X \) is a diffeological space and let \( \{ \phi_i \colon U_i \to X  \}_{i\in I}\) be the set of all plots on \( X \). 
Consider the function:
\[ \bigsqcup_{i\in I} \phi_i \colon  \bigsqcup_{i\in I} U_i \to X \]
This function is a subduction since every plot on \( X \) factors through \( \bigsqcup\limits_{i\in I} \phi_i\) in an canonical way. 
\end{example}
In some cases, one requires maps with a greater degree of regularity than a subduction. 
For example, one might wish to generalize the notion of a submersion of manifolds to the diffeological setting. We saw above that a subduction between manifolds is not quite the same thing as a submersion.
This leads us to the notion of a local subduction. 
\begin{definition}\label{definition:local.subduction}
We say that \( f \) is a \emph{local subduction} 
if for all plots \( \phi \colon U_\phi \to Y \) and points \( x \in X \), \( u \in U_\phi \) such that \( \phi(u) = f(x) \), 
there exists an open neighborhood \( V \subset U_\phi \) of \( u \in U_\phi \) and a plot \( \til \phi \colon U_\phi \to X \) 
such that \( \til \phi(u) = x \) and \( \phi|_V = f \circ \phi \).
\[
\begin{tikzcd}
 & (X,x) \arrow[d, "f"] \\
(U_\phi, u) \arrow[r, "\phi"]  & (Y, \phi(u) )
\end{tikzcd}
\quad \Rightarrow \quad
\begin{tikzcd}
\exists (V, u) \arrow[d, "\iota", hook, dashed] \arrow[r, " \exists\til \phi", dashed] & (X,x) \arrow[d, "f"] \\
(U_\phi, u) \arrow[r, "\phi"]  & (Y, \phi(u) )
\end{tikzcd}
\]
\end{definition}
The definitions of local subduction and subduction look similar. The main distinction, for subductions, is that the lift of \( \phi \) is not required to factor through any \emph{specific} point in \( X \). 
On the other hand, for a local subduction, one must be able to find a lift through every point in \( X \) that is in the preimage of \( \phi(u) \).
This has two main consequences: One is that being a local subduction is a property that is local in the \emph{domain} of \( f \). 
The other consequence is that, unlike subductions, local subductions may not necessarily be surjective. 
\begin{example}\label{example:local.subduction.manifolds}
If \( f \colon M \to N \) is a morphism of smooth manifolds then \( f \) is a local subduction if and only if \( f \) is a (not necessarily surjective) submersion. 
\end{example}
\begin{example}\label{example:local.subduction.counter.example}
Consider \( f \colon \RR \sqcup \RR \to \RR \) where \( f(x) = 0 \) on the first connected component and \( f(x) = x \) on the second connected component. 
Then \( f \) is a subduction but \( f \) is not a local subduction. 
\end{example}
Our next example/lemma is one of particular interest to us. Arguments similar to the one below will be used frequently in this article.
\begin{lemma}\label{lemma:lie.groupoid.quotient.is.subduction}
Suppose \( \CG \grpd M \) is a Lie groupoid and let \( X = M / \CG \) be the set of orbits. Then quotient map \( \pi \colon M \to X \) is a local subduction.
\end{lemma}
\begin{proof}
    Suppose \( \phi \colon U_\phi \to X \) is a plot and let \( u \in U_\phi \) and \( p \in M \) be fixed such that \( \phi(u) = \pi(p) \).

    Since \( \pi \colon M \to X \) is a subduction, we know there exists an open neighborhood \( V \subset U_\phi \) of \( U \) and a smooth function \( \til \phi \colon V \to M \) with the property that \( \pi \circ \til \phi = \phi \).

    Let \( q := \phi(u) \in M \). Since \( p\) and \( q \) lie in the same \( \pi \) fiber, there exists a groupoid element \( g \in \CG \) with the property that \( s(g) = q\) and \( t(g) = p\).
    Let \( \sigma \colon \CO \to \CG \) be a local section of the source map defined in a neighborhood of \( \CO \subset M\) of \( q \) such that \( \sigma(q) = g \).

    Now let:
    \[ \overline{\phi}(v) := (t \circ \sigma \circ \til \phi)(v) \]
    where \( t \) is the target map for the groupoid. 
    If necessary, we shrink the domain \( V \) of \( \overline{\phi} \) to a smaller open neighborhood of \( u \) to ensure that \( \overline{\phi}\) is well-defined.

    Then a direct computation shows that \( \pi \circ \overline{\phi} = \phi \) and also \( \overline{\phi}(u) = p\). We conclude that \( \pi \) is a local subduction.
    \end{proof}
\subsection{Standard constructions}\label{subsection:standard.constructions}
The flexibility of diffeological spaces means that we have many tools for constructing new diffeological spaces out of old ones. 
Let us briefly review a few of the standard diffeological constructions that we will require. 
\begin{definition}\label{definition:intersection.diffeology}
Suppose \( X \) is a set and \( \{ \BD_i \}_{i \in I} \) is an arbitrary collection of diffeological structures on \( X \). 
The \emph{intersection diffeology} relative to this collection is the diffeology on \( X \) which declares a function \( \phi \colon U \to X \) to be a plot if and only if \( \phi \) is a plot in \( \BD_i \) for all \( i \in I \).

Suppose \( X \) is a set and we have an arbitrary collection \( C = \{ \phi \colon U_\phi \to X \} \) of set theoretic functions where the domains of the elements of \( C \) are Euclidean spaces. 
The \emph{diffeology generated by \(C \)} is the diffeology obtained by taking the intersection of all diffeologies which for which every element in \( C \) is a plot.
\end{definition}
Since the constant plots are contained in every possible diffeology, the intersection diffeology always exists. 
Since every diffeology contains the discrete diffeology, the intersection of an arbitrary collection of diffeologies is never empty. 
Many diffeologies can be constructed by taking intersections. 
Among these, the quotient diffeology is of particular importance. 
\begin{definition}\label{definition:quotient.diffeology}
Suppose \( X \) is a diffeological space and \( \sim \) is an equivalence relation on \( X \). 
Let \( X/ \sim \) be the set of equivalence classes and let \( \pi \colon X \to X/\sim \) be the canonical surjective function. 
The \emph{quotient diffeology} on \( X/\sim \) is the intersection of all diffeologies on \( Y \) which makes \( \pi \) a subduction.
\end{definition}
Surjective subductions and quotient diffeologies are in one-to-one correspondence. 
By this we mean that a surjective function \( f \colon X \to Y \) of diffeological spaces is a subduction if and only if the diffeology on \( Y \) is the quotient diffeology relative to the equivalence relation given by the fibers of \( f \). 

The category of diffeological spaces is fairly well-behaved. In particular it admits finite products and exponential objects.
\begin{definition}\label{definition:product.diffeology}
Given two diffeological spaces, \( X \) and \( Y \) the \emph{product diffeology} on \( X \times Y \) is constructed as follows: We say that \( \phi \colon U_\phi \to X \times Y \) is a plot if and only if \( \pr_1 \circ \phi \) and \( \pr_2 \circ \phi \) are plots. 
\end{definition}
\begin{definition}\label{definition:functional.diffeology}
    Given two diffeological spaces \( X \) and \( Y \) the \emph{functional diffeology} on \( C^\infty(X,Y) \) is constructed as follows: We say that \( \phi \colon U_\phi \to C^\infty(X,Y) \) is a plot if and only if the following map is smooth:
    \[ \overline{\phi} \colon U_\phi \times X \to Y \qquad (u,x) \mapsto \phi(u)(x) \]
    Note that we take the product diffeology on \( U_\phi \times X \).
\end{definition}
\section{Quasi-étale diffeological spaces}\label{section:quasi-étale.diffeological.spaces}
We will now introduce a new class of diffeological space that we will call ``quasi-étale diffeological spaces" or QUED for short. This subcategory of diffeological spaces, along with its associated groupoids, will be the main focus of the remainder of the article.
In short, a QUED is a diffeological space that can be expressed as the (local) quotient of a smooth (Hausdorff) manifold by a well-behaved equivalence relation. 
The term "quasi-étale" is used because this equivalence relation causes the quotient map to behave in a manner that is reminiscent of an étale map of manifolds. However, unlike a typical étale map, it is not a local diffeomorphism.
\subsection{Quasi-étale maps and spaces}
\begin{definition}[Quasi-étale]\label{defn:quasi-étale}
Suppose \( X \) and \( Y \) are diffeological spaces. A map \( \pi \colon X \to Y \) is said to be \emph{quasi-étale} if it satisfies the following properties:
\begin{enumerate}[(QE1)]
    \item \( \pi\) is a local subduction,
    \item the fibers of \( \pi \) are totally disconnected,
    \item if \( \CO \subset X \) is an open subset and we are given a smooth map \( f \colon \CO \to X \) with the property that \( \pi \circ f = \pi \) then \( f \) is a local diffeomorphism.
\end{enumerate}
Given \( x \in X\) a \emph{quasi-étale chart around \( x \)} consists of a quasi-étale map \( \pi \colon M \to X \) where \( M \) is a smooth manifold and such that \( x \) is in the image of \( \pi \).
We say that \( X \) is a \emph{quasi-étale diffeological space (QUED)} if for all \( x \in  X\) there exists a quasi-étale chart around \( x \). 

We write \( \QUED \) to denote the full subcategory of \( \Diffgl\) that consists of quasi-étale diffeolgical spaces.
\end{definition}
Since local subductions are open, every quasi-étale diffeological space is locally the quotient of a smooth manifold modulo an equivalence relation. In practice, it can be difficult to determine whether an equivalence relation on a manifold gives rise to a quasi-étale chart and typically the most difficult step to prove is (QE3).

However, there are a lot of interesting examples of Quasi-étale diffeological spaces. In particular, several kinds of diffeological spaces that appear in the literature happen to be quasi-étale.
\begin{example}\label{example:quasi-étale.manifolds}
Suppose \( N \) and \( M \) are smooth manifolds. 
A quasi-étale map \( \pi \colon M \to N \) is the same thing as an étale map.
Therefore a classical atlas on a manifold is also a quasi-étale atlas.
\end{example}
\begin{example}\label{example:quasi-étale.quasifolds}
A quasifold (introduced by Prato~\cite{prato_sur_1999}) is a diffeological space that is locally the quotient of Euclidean space modulo a countable group of affine transformations.

If \( \Gamma \) is a countable group of affine transformations of \( \RR^n \) then the quotient map \( \RR^n \to \RR^n/\Gamma \) is quasi-étale. It is a local subduction because the quotient map for a smooth group action is always a local subduction. The fibers are totally disconnected since \( \Gamma \) is countable.
The fact that condition (QE3) is satisfied in this context is actually rather non-trivial but it has already been proved by Karshon and Miyamoto~\cite{karshon_quasifold_2022} (Corollary 2.15).
\end{example}
\begin{example}\label{example:quasi-étale.orbifolds}
Since orbifolds are a special case of quasifolds, they are also quasi-étale diffeological spaces.
\end{example}
\begin{example}\label{example:quasi-étale.diffeological.étale.manifolds}
In the literature, the closest definition to the one we give above is probably the ``diffeological étale manifolds'' of Ahmadi~\cite{ahmadi_submersions_2023}. In the Ahmadi article, he defines a class of maps which he calls ``diffeological étale maps.'' 

We will not state the definition of such maps here but we will simply utilize some of the properties proved by Ahmadi to show the relationship with our notion: Ahmadi shows that if \( \pi \colon X \to Y \) is a diffeological étale map then it is a local subduction and it has the property that for all \( x \in X \): 
\[
    T^{int}_x \pi \colon T^{int}_x X \to T^{int}_{\pi(x)} Y   
\]
is a bijection (\cite{ahmadi_submersions_2023}, Corollary 5.6(i)). 
Here \( T^{int} \) denotes the ``internal tangent space'' (see \cite{hector_geometrie_1995} and \cite{christensen_tangent_2016}.) If \( X \) is a smooth manifold and we have an \( f \) as in (QE3) such that \( \pi \circ  f  = \pi \) we can apply the tangent functor to this equation to easily conclude that \( f \) must be a local diffeomorphism.
\end{example}
Our last example is of particular importance.
It says that quotients of Lie groups by totally disconnected subgroups are quasi-étale diffeological spaces.

The proof strategy for this lemma is essentially a simplified form of the main proof strategy we will use to show that the Ševera-Weinstein groupoid of a Lie algebroid is a quasi-étale diffeological space.
\begin{lemma}\label{lemma:homogeneous.space}
Suppose \( G \) is a Lie group and \( K \) is a totally disconnected normal subgroup of \( G \). Let \( X = G/K\) be the group of \( K \)-cosets equipped with the quotient diffeology. The quotient map \( \pi \colon G \to X \) is quasi-étale and so \( X \) is a quasi-étale diffeological space.
\end{lemma}
\begin{proof}
We need to show each of the three properties.
(QE1) Consider the action groupoid \( K \ltimes G \grpd G \). Clearly \( \pi \) is the quotient map for the orbit space of this Lie groupoid. By Lemma~\ref{lemma:lie.groupoid.quotient.is.subduction} we conclude \( \pi \) is a local subduction.

(QE2) The fibers of \( \pi \) are the \( K \) cosets of the form \( gK \) for \( g \in G \). Since left translation by \( g \) is a diffeomorphism and \( K \) is totally disconnected, it follows that \(g K \) must be totally disconnected. 

(QE3) Suppose \( \CO \subset G \) is a connected open set and \( f \colon \CO \to G \) is a smooth function such that \(  \pi \circ f = \pi \).

Consider the function:
\[ \overline{f} (g) := g\inv \cdot f(g)  \]
Since \( \pi \) is a homomorphism it is clear that the function \( \pi \circ \overline{f} \colon \CO \to X \) is constant and therefore the image of \( \overline f \) is contained in a single fiber of \( \pi \). Since the fibers of \( \pi \) are totally disconnected, we conclude that there exists \( g_0 \in G\) such that \( \overline{f}(g) = g_0 \) for all \( g \in G \). 

This implies that 
\[ f(g) = g \cdot g_0  \]
Since \( f \) is just right translation by \( g_0 \) we conclude that \( f \) is a local diffeomorphism.
\end{proof}
\subsection{Properties of Quasi-étale maps}
Quasi-étale diffeological spaces have a variety of favorable properties that makes it possible to transport many differential geometry techniques to the category \( \QUED \).
The way one does this is by ``representing'' maps between quasi-étale diffeological spaces using maps between manifolds.
This is analogous to how smooth maps between manifolds can be represented in terms of charts. 
\begin{definition}\label{definition:local.representation}
Suppose \( f \colon X \to Y \) is a smooth map in \( \QUED \). 
A \emph{local representation} of \( f \) consists of a pair of quasi-étale charts \( \pi_X \colon M \to X \) and \( \pi_Y\colon N \to Y \) together with a smooth function \( \til f \colon M \to N \) such that the following diagram commutes:
\begin{equation}\label{diagram:local.representation}
\begin{tikzcd}
M \arrow[d, "\pi_X"] \arrow[r, "\til f"] & N \arrow[d, "\pi_Y"]  \\
X \arrow[r, "f"] & Y
\end{tikzcd}
\end{equation}
We say that a local representation of \( f \) is \emph{around \( x \in X \)} if \( x \) is in the image of \( \pi_X \). 
\end{definition}
Our next lemma states that representations of smooth maps in \( \QUED \) always exist.
\begin{lemma}\label{lemma:representations.exist}
    Suppose \( f \colon X \to Y \) is a smooth in \( \QUED \). Then for all \( x \in X \) there exists a local representation of \( f \) around \( x \in X \).
\end{lemma}
\begin{proof}
Since \( Y \) is quasi-étale, there must exist a quasi-étale chart \( \pi_Y \colon N \to Y \) around \( f(x) \in Y \).
Now let \( \pi_X \colon M' \to X \) be a quasi-étale chart around \( x \in X \). 
Since \( \pi_Y \) is a subduction, we know there must exist a local lift \( \til f \colon M \to N \) of of \( f \circ \pi_X \colon M' \to Y\) defined on some open subset \( M \subset M'\). 
The triple \( \til f, \pi_X|_{M}\) and \( \pi_Y \) is the desired local representation.
\end{proof}
Our next lemma observes that any two points in the same fiber of a quasi-étale chart can be related by a diffeomorphism.
\begin{lemma}\label{lemma:fiber.hopping}
    Suppose \( X \) is a quasi-étale diffeological space and \( \pi \colon M \to X \) is a quasi-étale chart. If \( p, q \in M \) are such that \( \pi(p) = \pi(q) \) then there exists a local diffeomorphism \( f \colon \CO \to M \) defined on a neighborhood \( \CO \subset M \) of \( p\) such that \( f(p) = q\) and \( \pi \circ f = \pi \).
\end{lemma}
\begin{proof}
    Since \( \pi \colon M \to X\) is a local subduction and \( M \) is a manifold, we know that there must exists an open neighborhood \( \CO \subset M  \) and a smooth function \( f \colon \CO \to M \) with the property that \( \pi \circ f = \pi \) and \( f(p) = q\).

    Since \( \pi \) is assumed to be quasi-étale, (QE3) implies that such an \( f \) must be a local diffeomorphism.
\end{proof}
Our next lemma relates properties of maps between quasi-étale diffeological spaces and properties of their representations.
\begin{proposition}\label{proposition:local presentations.of.morphisms}
Suppose \( f \colon X \to Y \) is a smooth map in \(\QUED\) and we have a local representation of \( f \) as in the Diagram~\ref{diagram:local.representation}.
\begin{enumerate}[(a)]
\item If \( f \) is a local diffeomorphism, then \( \til f \) is a local diffeomorphism.
\item If \( f \) is a local subduction then \( \til f \) is a submersion.
\item If \( f \) is constant, then \( \til f \) is locally constant.
\end{enumerate}
\end{proposition}
\begin{proof}
Throughout, let \( p \in M\) be arbitrary and \( q := \til f(p) \in N \), \( x := \pi_X(p) \in X \) and \( y := \pi_Y(q) \in Y\).

(a) Since \(N \) is a smooth manifold and \( f \circ \pi_X \) is a local subduction, there must exists a smooth function \( g \colon \CO \to M \) defined on an open neighborhood \( \CO \) of \( q\) such that 
\[f \circ \pi \circ g = \pi_Y \]
This implies that:
\[ \pi_Y \circ \til f \circ g = \pi_Y\]
Since \( \pi_Y\) is quasi-étale, by (QE3) it follows that \( \til f \circ g \) is a local diffeomorphism. Therefore, it follows \( \til f \) is a submersion. and the dimension of \( M \) is greater than or equal to the dimension of \( N \). 

(b) Now suppose \( f \) is a diffeomorphism. Then it is, in particular, a local subduction. 
Let \( g \) be as in the discussion for part (a). 
We can repeat the argument from part (a) where we replace \( f \) with \( f \inv\) and \( \til f \) with \( g \) and from that we conclude that \( g \) is a submersion.
Since \( \til f \circ g \) is a local diffeomorphism and both \( f \) and \( g \) are submersions we conclude that \( \til f \) is a local diffeomorphism.

(c) Since \( f \) is a constant function, it follows that the image of \( \til f \) is contained in \( \pi_Y\inv(y) \).
By assumption, we know that \( \pi_Y\inv(y)\) is totally disconnected so \( f \) is locally constant.
\end{proof}
We conclude this subsection with a lemma that simplifies the process of proving a map is a quasi-étale chart when we already know the space is quasi-étale.
\begin{lemma}\label{lemma:finding.quasi.étale.charts}
Suppose \( X \) is a quasi-étale diffeological space and suppose \( \pi \colon M \to X \) is a local subduction and the fibers of \( \pi \) are totally disconnected. Then \( \pi \) is a quasi-étale chart.
\end{lemma}
\begin{proof}
    Let \( p_0 \in M \) be an arbitrary point and let \( x_0 := \pi(p_0) \).
    We will show that \( \pi \) is quasi-étale in a neighborhood of \( p \). 
    
    Take \( \pi' \colon N \to X \) be a quasi-étale chart around \( x_0 \) and let \( q_0 \in N \) be such that \( \pi'(q_0) = x_0\)
    Since both \( \pi \) and \( \pi' \) are local subductions, we know that by possibly shrinking \( M \) and \( N \) to smaller open neighborhoods of \( p_0 \) and \( q_0 \), respectively, we can construct a pair of smooth fuctions \( f \colon M \to N \) and  \( g \colon N \to M \) such that \( f(p_0) = q_0\), \( g(q_0) = p_0 \), \( \pi' \circ f = \pi\), and \( \pi \circ g = \pi' \).

    Then it follows that \( f \circ g \colon N \to N \) is a smooth function such that \( \pi' \circ (f \circ g) = \pi' \).
    Since \( \pi' \) is quasi-étale, it follows that \( f \circ g \) is a local diffeomorphism. Therefore, it follows that \( f \) is a submersion. However, since \( \pi' \circ f = \pi \) the fibers of \( f \) must be totally disconnected and so \( f \) is a local diffeomorphism.

    Since \( f \colon M \to N \) preserves the projections to \( X \), and \( \pi' \colon N \to X \) is quasi-étale, it follows that \( \pi \) must also be quasi-étale.
\end{proof}
\subsection{Fiber products}
It is a notable property of diffeological spaces that the fiber product of two diffeological spaces always exist. However, fiber products of quasi-étale diffeological spaces may not be quasi-étale. This can happen even when the maps involved are quasi-étale.
Consider the following simple example:
\begin{example}\label{example:non.existant.fiber.product}
    Consider the \( \ZZ_2\) action on \( \RR \) by the reflection \( x \mapsto - x \). Let \( X := \RR/ \ZZ_2\). Then \( X \) is quasi-étale (it is an orbifold).
    However, consider the fiber product:
    \[ \RR \times_X \RR = \{ (x,y) \in \RR^2 \ : \ x = y \mbox{ or } x = -y \} \]
    The diffeological space \( \RR \times_X \RR \) is not a quasi-étale diffeological space. To see why, note that the domain of a quasi-étale chart around \( (0,0) \) would have to be to be one-dimensional. It is an standard exercise to show that there cannot exist a local subduction \( \pi \colon \RR \to \RR \times_X \RR \) with \( (0,0) \) in the image of \( \pi \).
\end{example}
What this tells us is that ``local subduction'' is not a strong enough condition to function as our notion of ``submersion'' between quasi-étale diffeological spaces.
In order to clarify what is going on here, will will need to use the notion of a ``cartesian'' morphism in a category.
\begin{definition}\label{definition:cartesian.morphism}
Suppose \( \BC \) is an arbitrary category. We say that a morphism \( p \colon X \to Z \) in \( \BC \) is \emph{cartesian} if for all other morphisms \( f \colon Y \to Z \) one of the following holds:
\begin{enumerate}[(1)]
\item The fiber product \( X \fiber{p}{f} Y \) exists. I.e. we have a pullback square:
\[
\begin{tikzcd}
    X \fiber{p}{f} Y \arrow[r, "\pr_1"] \arrow[d, "\pr_2"] & X \arrow[d, "p"] \\
    Y \arrow[r, "f"] & Z
\end{tikzcd}
\]
\item The collection of commutative squares of the following form is empty:
\[
\begin{tikzcd}
W \arrow[r] \arrow[d] & X \arrow[d, "p"] \\
Y \arrow[r, "f"] & Z
\end{tikzcd}
\]
\end{enumerate}
\end{definition}
\begin{example}
    Suppose \( \BC \) is the category of sets \( \Set \). Then every function in \( \Set\) is cartesian. In this category case (2) in the above definition never occurs since the empty set \( \empty \) is an initial object in \( \Set\).
\end{example}
\begin{example}
    Consider the category of smooth manifolds \( \Man\). 
    Any submersion \( p \colon M \to N \) in \( \Man \) is cartesian.

    Observe that if we permit the empty set \( \empty \) to be an initial object in \( \Man \) condition (2) above never occurs. However, we will proceed with the convention that the empty set is not a manifold.
\end{example}
\begin{example}
    In the category of diffeological spaces \( \Diffgl \) every morphism is cartesian.
\end{example}
For a general morphism in \( \QUED \) there does not appear to be a simple geometric criterion for determining when a morphism is cartesian.
However, later in this section we will state such a criterion for the special case where the co-domain is a manifold.

We use the notion of a cartesian morphism to define a class of maps in \( \QUED \) that is well-behaved for fiber products and extends the classical notion of a submersion.
\begin{definition}
    Suppose \( p \colon X \to Z \) is morphism in \( \QUED \).   We say that \( p \) is a \emph{\( \QUED \)-submersion} if \( p \) is a local subduction and \( p \) is cartesian in \( \QUED \).
\end{definition}
In our theory of quasi-étale diffeological spaces, \( \QUED\)-submersions will play the role of submersions in the theory of manifolds. Note that from Example~\ref{example:non.existant.fiber.product} we know that quasi-étale maps are \emph{not} a special case of \( \QUED\)-submersions.

Our next lemma says that \( \QUED\)-submersions are stable under taking base changes.
It is consequence of the more general fact that the properties of being a local subduction and being cartesian are both stable under base changes. We will include a proof, regardless, for completeness.
\begin{lemma}
    Suppose \( p \colon X \to Z \) is a \( \QUED\)-submersion and \( f \colon Y \to Z \) is any other morphism. The base change of \( p \) along \( f \):
    \[ \pr_2 \colon X \times_Z Y \to Y  \]
    is a \( \QUED\)-submersion.
\end{lemma}
\begin{proof}
Suppose \( p \colon X \to Z \) is a \( \QUED \) submersion and \( f \colon Y \to Z \) is a morphism in \( \QUED \). Let us illustrate the situation with a diagram:
\[
\begin{tikzcd}
X \times_Z Y \arrow[r, "\pr_1"] \arrow[d, "\pr_2"] & X \arrow[d, "p"] \\
Y \arrow[r, "f"] & Z 
\end{tikzcd}
\]
We need to show that:
\[  ( X \times_Z Y ) \times_Y W \]
is a quasi-étale diffeological space.
Observe that \( ( X \times_Z Y ) \times_Y W \) is canonically diffeomorphic to \( X \times_Z W \) as a diffeological space. Since \( p \) is cartesian we conclude that \( X \times_Z W \) is a quasi-étale diffeological space so \( \pr_2 \colon X \times_Z Y \to Y \) is cartesian.

Now we must show that \( \pr_2 \colon X \times_Z Y \to Y \) is a local subduction.
Suppose \( \phi \colon U_\phi \to Y \) is a plot and let \( u \in U_\phi \) and \( (x,y) \in X \times_Z Y \) be such that \( \phi(u) = y\).
Since \( p \colon X \to Z \) is a local subduction, we know there exists an open neighborhood \( V \subset U_\phi \) of \( u \) and a lift \( \psi \colon V \to X \) with the property that \( \psi(u) = x\) and \( p \circ \psi = f \circ \phi|_V \).
Then 
\[\til \phi := \phi|_V \times \psi \colon V \to X \times_Z Y  \]
is well-defined and has the properties that \( \pr_2 \circ \til \phi = \phi|_V \) and \( \til \phi(u) = (x,y)\).
\end{proof}
It is not that easy to construct ``obvious'' examples of \( \QUED\)-submersions. It is essential to our theory that \( \QUED\)-submersions between manifolds be the same thing as ordinary submersions. However, the proof of this fact is not as obvious as one might expect.
\begin{example}\label{example:submersion.is.qeds.submersion}
    Suppose \( p \colon M \to N \) is a smooth map of manifolds. Then \( p \) is a \( \QUED \)-submersion if and only if \( p \) is a submersion.

    One of the directions in the above statement is clear from the fact that local subductions between manifolds are automatically submersions. 
    The other direction follows from Theorem~\ref{theorem:qeds.submersion.criteria} which we will state later in this subsection.
\end{example}
\subsection{Fiber products over manifolds}
In this last subsection we will give some statements about the behavior of fiber products along \( \QUED \)-submersions in the special case where the fiber product is taken over a smooth manifold.
\begin{proposition}\label{prop:fiber.product.chart}
    Suppose \( p \colon X \to B \) is a local subduction where \( B \) is a smooth manifold. Suppose we have another morphism \( f \colon Y \to B \) in \( \QUED\) for which \( X \times_B Y\) is also a quasi-étale diffeological space. 
    
    Then for any point \( (x_0,y_0) \in X \times_B Y \). 
    If \( \pi_X \colon M \to X \) and \( \pi_Y \colon N \to Y \) are quasi-étale charts around \( x \) and \( y \) respectively, then the map:
    \[ \pi_{XY} \colon M \times_B N \to X \times_B Y \qquad \pi_{XY}(m,n) := (\pi_X(m), \pi_Y(n))  \]
    is a quasi-étale chart around \( (x_0, y_0) \).
\end{proposition}

\begin{proof}
    We have a commutative diagram:
        \[
    \begin{tikzcd}
    M \arrow[dd, "\pr_2"] \arrow[rr, "\pr_1"] \times_B N \arrow[dr, "\pi_{XY}"] & & M \arrow[d, "\pi_X"] \\
        &    X \times_B Y \arrow[r, "\pr_1"] \arrow[d, "\pr_2"] & X \arrow[d, "p"] \\
    N \arrow[r,"\pi_Y"]     &   Y \arrow[r, "f"] & B
    \end{tikzcd}
    \]
   Note that \( M \times_{B} N \) is a manifold since \( p \circ \pi_X\) is a local subduction (and hence a submersion). We must show that \( \pi_{XY} \) is quasi-etale. By Lemma~\ref{lemma:finding.quasi.étale.charts}, it suffices to show that \( \pi_{XY} \) satisfies (QE1) and (QE2) from Definition~\ref{defn:quasi-étale}.

    (QE1) We need to show that \( \pi_{XY}\) is a local subduction. 
    Suppose \( \phi \colon U_\phi \to X \times_B Y\) is a plot. 
    Let \( u \in U_\phi \) and \( (m,n) \in M \times_B N \) be such that:
    \[\phi(u) = \pi_{XY}(m,n)  \]
    
    We may split \( \phi \) into a product \( \phi = \phi_1 \times \phi_2 \) where \( \phi_1\) is a plot on \( X \) and \( \phi_2 \) is a plot on \( Y \).
    Since \( \pi_X\) and \( \pi_Y\) are local subductions, it follows that there exist an open neighborhood \( V \subset U_\phi \) of \( u \) and plots:
    \[ \til \phi_1 \colon V \to M \qquad  \til \phi_2 \colon V \to N \] 
    on \( M\) and \( N\) respectively such that:
    \[\pi_X \circ \til \phi_1 = \phi_1|_V \qquad  \pi_Y \circ \til \phi_2 = \phi_2|_V \] 
    and 
    \[\til \phi_1(u) = m \qquad \til \phi_2(u) = n \]
    This pair of plots defines a plot \( \til\phi := \til \phi_1 \times \til \phi_2\) on \( M \times_B N\). Furthermore \( \pi_{XY} \circ \til \phi = \phi\) and \( \til \phi(u) = ( m,n) \). 
    This shows \( \pi_{XY}\) is a local subduction.

    (QE2) Next, we show that \( \pi_{XY}\) has totally disconnected fibers.
    Suppose 
    \[\phi \colon U_\phi \to M \times_B N \] 
    is a plot such that the image of \( \phi \) is contained in \( \pi_{XY}\inv(x,y) \) for some \( (x,y) \in X \times_B Y \).
    Then \(\pr_1 \circ \phi \) is a plot with image contained in \( \pi_X\inv(x) \) and \( \pr_2 \circ \phi\) is a plot with image contained in \( \pi_Y\inv(y) \). Since \( \pi_X\) and \( \pi_Y\) have totally disconnected fibers, it follows that \( \phi \) is locally constant. \( \phi \) was arbitrary we conclude that \( \pi\inv_{XY}(x,y)\) is totally disconnected.
\end{proof}
The next theorem is our main result of this section. It says essentially that a smoothly parameterized family of quasi-étale diffeological spaces has a quasi étale total space.
\begin{theorem}\label{theorem:qeds.submersion.criteria}
    Suppose \( p \colon X \to B \) is a morphism in \( \QUED \) with \( B\) a smooth manifold and \( p \) a local subduction.
    Then \( p \) is a \( \QUED\)-submersion if and only if the fibers of \( p \) are quasi-étale.
\end{theorem}
This somewhat innocent looking theorem has a surprisingly involved proof.

Before getting to the proof of this theorem. We will first prove a lemma that simplifies the process of determining whether or not a map is quasi-étale. Note that the only difference is a subtle change in the last axiom in the definition of quasi-étale.
\begin{lemma}\label{lemma:quasi.étale.alternate}
Suppose \( \pi \colon M \to X \) is a smooth map of diffeological spaces with \( M \) a smooth manifold. Then \( \pi \) is quasi-étale if and only if the following conditions hold:
\begin{enumerate}[(1)]
    \item \( \pi \) is a local subduction,
    \item the fibers of \( \pi \) are totally disconnected,
    \item for all \( p \in M \) and smooth functions:
    \[ f \colon \CO \to M  \]
    such that \( \CO \) is an open neighborhood of \( p \), \( f(p)= p\)  and \( \pi \circ f = \pi\) we have that \(f \) is a diffeomorphism in a neighborhood of \( p \).
\end{enumerate}
\end{lemma}
\begin{proof}
    Note that the only difference between the two conditions above and the definition of quasi-étale is the additional stipulation that \( f \) has a fixed point. Therefore, the \( \Rightarrow \) case is clear.

    Now, suppose we have a map \( \pi \colon M \to X  \) which satisfies properties (1-3). We must show that \( \pi \) satisfies (QE3).
    Consider a smooth map \( f \colon \CO \to Y  \) where \( \CO \subset X \) is an open subset and \( \pi \circ f = \pi\).
    To show that \( f \) is a local diffeomorphism, it suffices to show that it is a local diffeomorphism in a neighborhood of an arbitrary \( p \in \CO \) so let \( p \in \CO \) be fixed.

    Since \( \pi \) is a local subduction and \( M \) is a smooth manifold, there exists an open neighborhood \( \CU \subset M \) of \( f(p) \) and a smooth function \( \psi \colon \CU \to M \) such that \( \psi(f(p)) = p \) and \( \pi \circ \psi = \pi\).

    We can assume without loss of generality that \( f\inv(\CU) \subset \CO \) and therefore the function \( \psi \circ f\) is well defined.
    This function also satisfies \( (\psi \circ f)(p) = p\) and \( \pi \circ (\psi \circ f) = \pi \). By property (3) above we conclude that \( \psi \circ f \) is a local diffeomorphism. Therefore, \( f \) must be an immersion. By a dimension count, we conclude that \( f \) is a local diffeomorphism.
\end{proof}
Our next lemma tells us that if we have a smoothly parameterized family of quasi-étale diffeological spaces with a quasi-étale total space, then one can find a quasi-étale chart on a fiber by restricting a quasi-étale chart on the total space.
\begin{lemma}\label{lemma:restriction.to.fiber}
    Suppose \( p \colon X \to B \) is smooth with \( X \) quasi-étale and \( B \) a smooth manifold and assume that for all \( b \in B\) 
    \[ X_b := p\inv(b) \]
    is quasi-étale.

    Then if \( \pi \colon M \to X \) is a quasi-étale chart, it follows that:
    \[ \pi|_{(p \circ \pi)\inv(b)}  \colon (p \circ \pi)\inv(b) \to X_b  \]
    is a quasi-étale chart.
\end{lemma}
\begin{proof}
    This is an immediate consequence of Proposition~\ref{prop:fiber.product.chart} where we take \( f \colon N \to B \) to be the inclusion of a point \( \{ b \} \into B \).
\end{proof}
With these observations, we can now prove the theorem.
\begin{proof}[Proof of Theorem~\ref{theorem:qeds.submersion.criteria}]
    Suppose \( B \) is a smooth manifold and let \( p \colon X \to B \) be a local subduction and a morphism in \( \QUED\). We must show that \( p \) is a \( \QUED\)-submersion if and only if the fibers of \( p \) are quasi-étale diffeological spaces.

    (\( \Rightarrow\)) Suppose that \( p \) is a \( \QUED\)-submersion. Then \( p \) must be cartesian and it follows that for all \( b \in B \) we have that the fiber product:
    \[ \{ b \} \times_{B} X \cong p\inv(b)  \]
    is a quasi-étale diffeological space.

    (\( \Leftarrow \) ) Suppose that \( p \colon X \to B \) is a local subduction in \( \QUED \) where \( B \) is a smooth manifold. Assume that for all \( b \in B \) we have that \( p\inv(b) \) is a quasi-étale diffeological space. 
    
    We must show that \( p \) is cartesian. Let \( f \colon Y \to B \) be an arbitrary smooth map in \( \QUED \). Our task is to show that \( X \times_B Y \) quasi-étale.
    
    Suppose \( \pi_X \colon M \to X \) and \( \pi_Y \colon N \to Y \) are quasi-étale charts.
    Let \( \pi_{XY} \colon M \times_B N \to X \times_B Y\) be the associated smooth map. We claim that \( \pi_{XY}\) is quasi-étale.

    First let us establish some notation and make a few observations: 
    \begin{itemize}
    \item Let 
    \[\til p := p \circ \pi_X \colon M \to B \qquad \til f := f \circ \pi_Y \colon N \to B  \]
    We remark that \( \til p  \) is a submersion between smooth manifolds.
\item Given \( b \in B \), let \( X_b := p\inv(b)\) and \( M_b := \til p\inv(b) \). Note that, by Lemma~\ref{lemma:restriction.to.fiber}, we know that 
    \[\pi_X|_{M_p} \colon M_p \to X_p \]
    is a quasi-étale chart.
    \item Let \( q := \pr_2 \colon M \times_B N \to N\). Note that since \( q \) is the base change of \( \til p \) along \( \til f \) it follows that \( q \) is a submersion.
    \end{itemize}

    We can illustrate the situation with the following diagram:
    \[
    \begin{tikzcd}
    M \arrow[dd, "q"] \arrow[rr, "\pr_1"] \times_B N \arrow[dr, "\pi_{XY}"] & & M \arrow[d, "\pi_X", swap] \arrow[dd, "\til p", bend left] \\
        &    X \times_B Y \arrow[r, "\pr_1"] \arrow[d, "\pr_2"] & X \arrow[d, "p", swap] \\
    N \arrow[rr, "\til f", bend right]\arrow[r,"\pi_Y"]     &   Y \arrow[r, "f"] & B
    \end{tikzcd}
    \]
    
    (QE1) We must show \( \pi_{XY} \) is a local subduction. Suppose \( \phi \colon U_\phi \to X \times_B Y \) is a plot and we have a point \( u \in U_\phi \) and \( (m,n) \in M \times_B N \) such that \( \pi_{XY}(m,n) = \phi(u) \). We know that there exist plots \( \phi_X \colon U_\phi \to X \) and \( \phi_Y \colon U_\phi \to Y \) such that:
    \[ \phi(u) = (\phi_X(u), \phi_Y(u)) \quad \mbox{ and } \quad p \circ \phi_X = f \circ \phi_Y  \]
    Since \( \pi_X \) and \( \pi_Y \) are local subductions, there must exist an open neighborhood \( V \) of \( u \in  U_\phi \) together with lifts \( \til \phi_X \) and \( \til \phi_Y \) such that: 
    \[\pi_X \circ \til \phi_X = \phi_X  \qquad \pi_Y \circ \til \phi_Y = \phi_Y \]
    and:
    \[  \til \phi_X(u) = m \qquad \til \phi_Y(u) = n \] 
    Therefore, we can define a map:
    \[ \til \phi \colon V \to X \times Y \qquad \til \phi(u) = (\til \phi_X(u), \til \phi_Y(u)) \]
    We claim that the image of \( \til \phi \) is contained in \( X \times_B Y \). This follows from a direct calculation:
    \begin{align*}
    \til p \circ \til \phi_X &= p \circ \pi_X \circ \til \phi_X \\
    &= p \circ \phi_X \\ 
    &= f \circ \phi_Y \\
    &= f \circ \pi_Y \circ \til \phi_Y \\
    &= \til f \circ \til \phi_Y
    \end{align*}
    Furthermore, it follows immediately from the definition of \( \pi_{XY} \) that \( \pi_{XY} \circ \til \phi = \phi \) and \( \til \phi(u) = (m,n) \). This shows that \( \pi_{XY} \) is a local subduction.

    (QE2) If \( \phi \colon U_\phi \to M \times_B N \) is a smooth map with image contained in the fiber of \( \pi_{XY}\), then it folows that \( \pr_1 \circ \pi_{XY} \circ \phi \) is locally constant. Commutativity of the top square implies that \( \pi_X \circ \pr_1 \circ \phi \) is locally constant. Since \( \pi_X \) is quasi-étale it follows that \( \pr_1 \circ \phi \) is locally constant. A symmetrical argument shows that \( \pr_2 \circ \phi \) is locally constant. Since each component of \( \phi \) is locally constant it follows that \( \phi \) is locally constant.
    
    (QE3) We utilize the simplification from Lemma~\ref{lemma:quasi.étale.alternate}.
    Let us fix \( (m_0,n_0) \in M \times_B N \) and suppose \( g \colon \CO \to M \times_B N \) is a smooth function such that \( \pi_{XY} \circ g = \pi_{XY} \) and \( g(m_0,n_0) = (m_0,n_0) \). We are finished if we can show \( g \) is a local diffeomorphism in a neighborhood of \( (m_0, n_0) \).

    Our strategy to show \( g \) is a local diffeomorphism involves several steps. This fact follows from the following sequence of claims:
    \begin{enumerate}[(1)]
        \item In a neighborhood of \( (m_0,n_0) \), \(g \) maps fibers of \( q \) to fibers of \( q \).
        \item Since \( \pi_Y \) is quasi-étale, the horizontal (relative to \( q \)) component of \( g \) is a local diffeomorphism.
        \item Since \( \pi_X|_{M_b}\) is quasi-étale for all \( b \in B \), it follows that the vertical component of \( g \) is a local diffeomorphism.
        \item Since both the horizontal and vertical components of \( g \) are local diffeomorphisms, it follows that \( g \) is a local diffeomorphism.
    \end{enumerate}
      
    (1) Since we are only concerned with the behavior of \( g \) in an open neighborhood of \( (m_0,n_0) \) and \( (m_0,n_0) \) is a fixed point of \( g \) we can freely restrict \( g \) to arbitrarily small neighborhoods of \( (m_0,n_0) \). Therefore, we can assume without loss of generality that there exist two open subsets \( \CU  \subset M\) and \( \CV \subset N \) such that \( \CO = \CU \times_B \CV \). Since \( \til p \) is a submersion, we can choose \( \CU \) in such a way that for all \( b \in B \) we have that:
    \[  M_b \cap \CU  \]
    is connected. 

    Note that if \( n \in N \) is a point such that \( \til f(n) = b\), it follows that:
    \[ q\inv(n) =  M_{b} \times \{ n \} \]
    Consequently, it follows that for all \( n \in N \) the sets:
    \[  q\inv(n) \cap \CO \]
    are connected.

    We claim that \( g \) must preserve the fibers of \( q \) on such a domain. To see why, observe that by following the commutative diagram above and the fact that \( \pi_{XY} \circ g = \pi_{XY}\) we have that:
    \[ \pi_Y \circ q \circ g = \pr_2 \circ \pi_{XY} \circ g = \pr_2 \circ \pi_{XY} = \pi_Y \circ q\]
    Since, for all \( n \in N \), the subset \( q\inv(n) \cap \CO \) are connected and the fibers of \( \pi_Y \) are totally disconnected, it follows that the image:
    \[q \circ g ( q\inv(n) \cap \CO  )   \]
    is a point. This implies that the image of the fibers of \( q \) under \( g \) are contained in fibers of \( q \). Since \( q \) is a submersion, it follows that there must exist a smooth function \( h \) which makes the following diagram commute:
    \[
    \begin{tikzcd}
        \CO \arrow[d, "q|_{\CO}"]  \arrow[r, "g"]& M \times_B N  \arrow[d, "q"] \\
        \CU \arrow[r, "h"] & N
    \end{tikzcd}
    \]
    
    (2) Note that since \( \pi_{XY} \circ g = \pi_{XY} \) it follows that \( \pi_Y \circ h = \pi_Y\). 
    Since \( \pi_Y \) is quasi-étale it follows that \( h \) is a local diffeomorphism.
    
    (3) For each \( n \in N \) let \( \til f(n) = b\). Note that:
    \[ q\inv(n) \cap \CO = (M_b \cap \CU)  \times \{ n \} \]
    Furthermore, we remark that \( \til f \circ h = \til f\) so it follows that:
    \[ g((M_b \cap \CU) \times \{ n \} ) \subset M_b \times \{ h(n) \} \]
    Therefore, let:
    \[ v_n \colon M_b \cap \CU \to M_b  \]
    the the map induced by \( g \) at the level of fibers. In other words, \( v_n \) is the vertical component of \( g \) at the \( q \) fiber of \( n \).

    Observe that since \( \pi_{XY} \circ g = \pi_{XY} \) it follows that \( \pi_X \circ v_b = \pi_X \).
    From a previous lemma, we saw that \( \pi_X |_{M_b} \) is quasi-étale for all \( b \in B \) whenever \( M_b \) is non-empty. From the quasi-étale property, it follows that \( v_b \) is a local diffeomorphism.

    (4) Since both \( h \) and \( v_b \) are local diffeomorphisms for all \( b \in B\) it follows that \( g \) is a local diffeomorphism.
\end{proof}
\section{Quasi-étale groupoids}\label{section:quasi.étale.groupoids}
We will now look at groupoid objects in the category of quasi-étale diffeological spaces. We will also need to define the notion of a local groupoid in this context. It turns out that
local groupoids will play a key role in constructing the differentiation functor.
\subsection{Groupoid objects in \texorpdfstring{\( \QUED \)}{QUED} }
\begin{definition}
A \emph{\(\QUED\)-groupoid} consists of two quasi-étale diffeological spaces \( \CG \) and \( X \) called the \emph{arrows} and \emph{objects} respectively together with a collection of smooth maps:
\begin{itemize}
\item Two \(\QUED\)-submersions called the \emph{source} and \emph{target}:
\[ \s \colon \CG \to X \qquad  \t \colon \CG \to X \]
\item A smooth map called the \emph{unit}
\[ \u \colon X \to \CG  \qquad x \mapsto 1_x  \]
\item A smooth map called \emph{multiplication}
\[
\m \colon \CG \fiber{s}{t} \CG \to \CG \qquad (g,h) \mapsto gh
\]
\item A smooth map called \emph{inverse}
\[
\i \colon \CG \to \CG 
\]
\end{itemize}
These morphisms are required to satisfy the following properties:
\begin{enumerate}[(G1)]
\item Compatibility of source and target with unit:
\[
\forall x \in M \qquad s(\u(x)) = t(\u(x)) = x
\]
\item Compatibility of source and target with multiplication:
\[
\forall (g,h) \in \CG \fiber{\s}{\t} \CG \qquad \s(m(g,h) = s(h) \quad \t( m(g ,h) = t(g)  
\]
\item Compatibility of source and target with inverse:
\[ \forall g \in \CG \qquad \s(\i(g)) = \t(g) \]
\item Left and right unit laws:
\[\forall g \in \CG \qquad m( \u(\t(g)), g) =  g = \m(g, \u(\s(g)))   \]
\item Left and right inverse laws:
\[ \forall g \in \CG \qquad \m(g, \i(g)) = \u(\t(g)) \quad \m(\i(g),g) = \u(\s(g))  \]
\item Associativity law:
\[ \forall (g,h,k) \in \CG \fiber{\s}{t} \CG \fiber{\s}{\t} \CG \qquad m(g,m(h,k)) = m(m(g,h),k)  \]
\end{enumerate}
We say that such a groupoid \( \CG \grpd X  \) is a \emph{singular Lie groupoid} if \( X \) is a smooth manifold. 
We say that \( \CG \grpd X \) is a \emph{Lie groupoid} if both \( \CG \) and \( X \) are manifolds.
\end{definition}
It should be noted that the condition that the source and target maps be \( \QUED \)-submersions is crucial. For example, this property ensures that the spaces of composable arrows are also objects in \( \QUED \).

Singular Lie groupoids are precisely the class of diffeological groupoids that we intend to differentiate to Lie algebroids.
\begin{definition}
Suppose \( \CG \grpd X \) and \( \CH \grpd Y \) are \( \QUED \)-groupoids. A \emph{groupoid homomorphism} \( F \colon \CG \to \CH \) covering \( f \colon X \to Y \) is a pair of smooth maps with the properties:
\begin{enumerate}
    \item Compatibility with source and target:
    \[ \forall g \in \CG \qquad \s(F(g)) = f(s(g)) \quad \t(F(g)) = f(t(g)) \]
    \item Compatibility with multiplication:
    \[ \forall (g,h) \in \CG \fiber{\s}{\t} \CG \qquad F(m(g,h)) = m(F(g),F(h))  \]
\end{enumerate}
\end{definition}
\begin{example}
Since a \( \QUED\)-submersion between smooth manifolds is an ordinary submersion, this definition of a Lie groupoid corresponds exactly to the classical definition. 
\end{example}
\begin{example}
Suppose \( G \) is a Lie group and \( K \) is a totally disconnected normal subgroup of \( G \).
According to Lemma~\ref{lemma:homogeneous.space} we know \( G/K \) is a quasi-étale diffeological space and therefore \( G/K\) is a singular Lie groupoid.
\end{example}
\begin{example}\label{example:singular.tr}
Consider the tangent bundle \( T \RR \) and think of it as a bundle of abelian groups over \( \RR \). Let us define an action of \( \ZZ \) on \( T \RR \) by the following rule:
\[ z \in \ZZ, \  (v,t) \in T\RR  \qquad     z \cdot (v,t)  := (v + tz, t)  \]
where the second coordinate is the base point and the first coordinate is the vector component. The orbit space of this action \( T \RR/\ZZ \) is quasi-étale so it is canonically a singular Lie groupoid over \( \RR \).
This example illustrates how quasi-étale groupoids are able to handle the ``transverse obstruction'' to integrability.
\end{example}
\subsection{Local groupoids}
Local groupoids are a weaker form of a groupoid which has a more restrictive product operation. Essentially, we do not require that multiplication be defined for all ``composable'' pairs. The ``local'' part of the definition of a local groupoid refers to the fact that multiplication should be defined in an open neighborhood of the units.
\begin{definition}
A \emph{local \(\QUED\)-groupoid} consists of a pair of quasi-étale spaces \( \CG \) and \( X \) called the \emph{arrows} and \emph{objects} together with:
\begin{itemize}
\item two \(\QUED\)-submersions called the \emph{source} and \emph{target}:
\[ \s \colon \CG \to X \qquad  \t \colon \CG \to X \]
\item a smooth map called the \emph{unit}
\[ \u \colon X \to \CG  \qquad x \mapsto 1_x  \]
\item an open neighborhood \( \CM \subset \CG \fiber{\s}{t} \CG \) of \( \u(M) \times \u(M)\) and smooth map called \emph{multiplication}
\[ 
\m \colon \CM \to \CG \qquad (g,h) \mapsto gh
\]
\item An open neighborhood \( \CI \subset \CG\) of \( u(M) \) and a smooth map called \emph{inverse}
\[ 
\i \colon \CI \to \CG  \qquad g \mapsto g\inv
\]
\end{itemize}
We further require that the following properties hold:
\begin{enumerate}[(LG1)]
\item Compatibility of source and target with unit:
\[
\forall x \in M, \qquad s(\u(x)) = t(\u(x)) = x
\]
\item Compatibility of source and target with multiplication:
\[
\forall (g,h) \in \CM, \qquad  \s(m(g,h) = s(h) \quad \t( m(g ,h) = t(g)  
\]
\item Compatibility of source and target with inverse:
\[ \forall g \in \CI, \qquad \s(\i(g)) = \t(g) \]
\item Left and right unit laws: \( \forall g \in \CG \)
\[ \forall g \in \CG \qquad m( \u(\t(g)), g) =  g = \m(g, \u(\s(g)))   \]
whenever the above expression is well-defined 
\item Left and right inverse laws:
\[ \forall g \in \CG, \qquad  \m(g, \i(g)) = \u(\t(g)) \quad \m(\i(g),g) = \u(\s(g))  \]
whenever the above expression is well-defined
\item Associativity law:
\[ \forall (g,h,k), \in \CG \fiber{\s}{t} \CG \fiber{\s}{\t} \CG \qquad m(g,m(h,k)) = m(m(g,h),k)  \]
whenever the above expression is well-defined.
\end{enumerate}
A local \( \QUED\)-groupoid \( \CG \grpd X \) is said to be a \emph{singular local Lie groupoid} if \( X \) is a smooth manifold. It is a \emph{local Lie groupoid} if both \( \CG \) and \( X \) are manifolds.
\end{definition}
Apart from the fact that multiplication is defined only on an open subset, for our purposes, local groupoids will have a distinction in how their morphisms are defined.

Let us establish some notation. Given a smooth function of diffeological spaces \( f \colon X \to Y \) and a pair of subsets \( S \subset X \) and \( T \subset Y \) such that \( f(S) \subset T \) we write:
\[ [f]_S \colon [X]_S \to [Y]_T \]
to denote the class of \( f \) as a germ of a map from \( X \) to \( Y\) defined in an open neighborhood of \( S \). The expressions \( [X]_S\) and \( [Y]_T\) denote germs of diffeological spaces.
\begin{definition}\label{definition:local.groupoid.morphism}
Suppose \( \CG \grpd X \) and \( \CH \grpd Y \) are local \( \QUED\)-groupoids. A \emph{morphism} \( \CG \to \CH \) covering \( f \colon X \to Y \) consists of a \emph{germ} of a smooth function:
\[ [\CF]_{u(M)} \colon [\CG]_{u(M)} \to [\CH]_{u(N)} \]
such that there exists a representative of the germ \( \CF \) which satisfies the following properties:
\begin{enumerate}
    \item Compatibility with source and target:
    \[ \forall g \in \CG \qquad \s(\CF(g)) = f(s(g)) \quad \t(\CF(g)) = f(t(g)) \]
    whenever the above expressions are well-defined.
    \item Compatibility with multiplication:
    \[ \forall (g,h) \in \CG \fiber{\s}{\t} \CG \qquad \CF(m(g,h)) = m(\CF(g),\CF(h))  \]
    whenever the above expression is well-defined.
\end{enumerate}
\end{definition}
It is built-in to our notation that \( \CF (u(M)) \subset u(N) \). In other words, a morphism of local \( \QUED\)-groupoids always maps units to units by definition.

A \( \QUED \)-groupoid can also be thought of as a local \( \QUED \)-groupoid. However, we remark that the functor from \( \QUED\)-groupoids to local \( \QUED \)-groupoids is not full or faithful. In other words, \( \QUED \)-groupoids do not form a subcategory of their local counterparts.

In order to reduce the potential of confusion, we will always use a symbol such as \( \CF \) to denote an actual smooth map compatible with the structure maps and \( [\CF] \) to denote the germ of such a map around the units. 
This is a mild abuse of notation since we should formally use \( [\CF]_{\u(M)} \) but we will omit the repetitive subscripts to reduce notational clutter.
    
Local Lie groupoids are particularly relevant to the integration problem due to the following theorem.
\begin{theorem}[\cite{crainic_integrability_2003},\cite{cabrera_local_2020}]\label{theorem.local.lie.groupoids.vs.algebroids}
The Lie functor for local groupoids defines an equivalence of categories:
\[ \Lie \colon \{\mbox{Local Lie groupoids}\} \to \mbox \{ \mbox{Lie algebroids}\} \]
\end{theorem}
The first proof of local integrability of Lie algebroids appears in Crainic and Fernandes~\cite{crainic_integrability_2003}. However, that the relationship is actually an equivalence of categories was, to some extent, folklore. Later, a complete proof of this equivalence appeared in \cite{cabrera_local_2020} and this is the earliest such proof that we are aware of.

This theorem means that we will not need to deal with Lie algebroids directly to define differentiation.
Instead, our strategy will be to construct a functor from singular Lie groupoids to local Lie groupoids.
The only other feature of Lie algebroids that we must keep in mind is that Lie algebroids are a \emph{sheaf}.
In other words, Lie algebroids can be constructed by gluing together Lie algebroids defined over an open cover together with coherent gluing data on the intersections. Since this functor is an equivalence, Local Lie groupoids are also a sheaf.

\section{Differentiation}\label{section:differentiation}
The key observation that we need for our differentiation procedure is that a quasi-étale chart (around an identity element) on a local singular Lie groupoid inherits a unique local groupoid structure compatible with the chart. In this sense, local Lie groupoids can be use to ``desingularize'' singular Lie groupoids.

This section will be organized as follows. In the first subsection we will state our basic theorems about quasi-étale charts on local singular Lie groupoids. In the next section we will show how these theorems can be used to construct a differentiation functor. In the last section we will use some of this theory to classify singular Lie groupoids with integrable algebroids.
\subsection{Representing singular Lie groupoids}
Our method of differentiating a singular Lie groupoid involves locally ``representing'' the said groupoid with a local Lie groupoid.
\begin{definition}\label{definition:local.lie.groupoid.chart}
Suppose \( \CG \grpd M \) is a local singular Lie groupoid. A \emph{local Lie groupoid chart} of \(\CG \grpd M \) consists of an open subset \( \til M \subset M \) together with a quasi-étale chart \( \pi \colon \til \CG \to \CG \) such that \( [\pi] \) constitutes a local groupoid morphism covering the inclusion:
\[
\begin{tikzcd}[column sep = large]
    \til \CG \arrow[r, "{[\pi]}" ]\darrow & \CG \darrow \\
    \til M \arrow[r, hook] & M
\end{tikzcd} 
\]
We say a local groupoid chart is \emph{around \( x \in M \)} if \( x \in \til M\). We say it is \emph{wide} if \( \til M = M \).
\end{definition}
The following two theorems are the main technical results that permit us to differentiate (local) singular Lie groupoids. The first is an existence and uniqueness result for local groupoid charts. The second theorem says that one can always represent a homomorphism of local singular Lie groupoids using local Lie groupoid charts.
\begin{theorem}\label{theorem:existence.and.uniqueness.of.LGC}
Suppose \( \CG \grpd M \) is a local singular Lie groupoid. There exists a wide local Lie groupoid chart \( \pi \colon \til \CG \to \CG \). Furthermore, if \( \pi' \colon \til \CG ' \to \CG \) is another wide local Lie groupoid chart there exists a unique isomorphism of local Lie groupoids \( [\CF ] \colon \til \CG \to \til \CG' \) which makes the following diagram commute:
\[\begin{tikzcd}
\til \CG \arrow[dr, "{[\pi]}", swap] \arrow[rr, "{[\CF]}"] & & {\til \CG' }\arrow[dl, "{[\pi']}"] \\
 & {[\CG] }&
\end{tikzcd}
\]
\end{theorem}
\begin{theorem}\label{theorem:existence.and.uniqueness.of.lifts}
Let \( \CG \grpd M \) and \( \CH \grpd N \) be local singular Lie groupoids.
Suppose \( \CF \colon \CG \to \CH \) is homomorphism and \( \pi_\CG \colon \til \CG \to \CG \) and \( \pi_\CH \colon \til \CH \to \CH \) are wide local groupoid charts. Then there exists a unique local groupoid morphism \( [\til \CF ] \colon \til \CG \to \til \CH \) which makes the following diagram commute:
\[
\begin{tikzcd}
\til \CG \arrow[r, "{[\til \CF]}"] \arrow[d, "{[\pi_{\CG}]}", swap] & \til \CH \arrow[d, "{[\pi_{\CH}]}"] \\
\CG \arrow[r, "{[\CF]}"] & \CH
\end{tikzcd}
\]
\end{theorem}
The idea of the proof is that, one starts with an arbitrary quasi-étale chart on \( \pi \colon \til \CG \to \CG \) and then constructs a local groupoid structure on \( \til \CG \) by finding representations of all of the groupoid structure maps.
Once these maps are constructed, one then needs to show that these representatives satisfy the groupoid axioms in a neighborhood of the identity.

The full proofs of Theorem~\ref{theorem:existence.and.uniqueness.of.LGC} and Theorem~\ref{theorem:existence.and.uniqueness.of.lifts} require their own section as well as some further technical development. Therefore, for the sake of exposition, we have moved these proofs to Section~\ref{section:proof.of.main.theorem}.

For now, let us explore some of the consequences of these theorems.
One interesting consequence is a classification of all connected singular Lie groups.
\begin{theorem}\label{theorem:Weinstein groups are quotients of lie groups}
Suppose \( G  \) is a connected singular Lie group. In other words, \( G \) is a singular Lie groupoid over a point. Then as a diffeological group, \( G \) is isomorphic to a quotient of a Lie group modulo a totally disconnected normal subgroup.
\end{theorem}
\begin{proof}
By Theorem~\ref{theorem:existence.and.uniqueness.of.LGC} we know that we can come up with a local Lie group \( \til G^\circ   \) together with quasi-étale chart \( \pi \colon \til G \to G \) which is a homomorphism of local groups.

By possibly shrinking \( \til G^\circ \) to a small enough open neighborhood of the identity, we can assume without loss of generality that \( \til G^\circ \) is an open subset of \( \til G \) where \( \til G \) is a simply connected Lie group.

Since the group \( \til G^\circ \) generates \( \til G \), it is possible\footnote{One way to prove this is by considering the subgroup of \( \til G \times H \) generated by the graph of \( \pi \). This can also be seen by using the theory of associative completions developed by Malcev\cite{malcev_sur_1941}. See also Fernandes and Michiels\cite{fernandes_associativity_2020} for a more modern version.} to extend \( \pi \) uniquely to a homomorphism defined on all of \( \til G \) and so we have a continuous group homomorphism \( \pi \colon \til G \to G \). Since \( \pi \) is quasi-étale in a neighborhood of the identity, it follows by a simple translation argument that it must be quasi-étale everywhere.
Since the fibers of a quasi-étale map must be totally disconnected it follows that the kernel of \( \pi \) is a totally disconnected normal subgroup of \( \til G \). Finally, since \( \pi \) is open, the image of \( \pi \) is an open subgroup. Since we have assumed that \( G \) is connected, it follows that \( \pi \) is surjective.
\end{proof}
\subsection{Construction of the Lie functor}
We will now prove the main theorem about differentiating singular Lie groupoids. Let us establish some notation for the relevant categories:
 \[ \Alg := \{ \text{Category of Lie algebroids} \} \]
 \[ \LocLieGrpd := \{ \text{Category of local Lie groupoids} \} \]
 \[ \SingLocLieGrpd := \{ \text{Category of singular local Lie groupoids} \} \]

Let us first restate the main theorem from the introduction:
\begin{reptheorem}{theorem:main.lie.functor}
    Let \( \SingLieGrpd \) be the category of \( \QUED\)-groupoids where the space of objects is a smooth manifold. There exists a functor:
\[ \hat \Lie \colon \SingLieGrpd \to \LieAlg  \]
with the property that \( \hat \Lie |_{\LieGrpd} = \Lie \).
\end{reptheorem}
In fact, we will prove a stronger result than the one stated above. The full version is stated for local singular Lie groupoids and includes a claim that that says the functor is essentially ``unique''. 
 \begin{theorem}\label{theorem:main.lie.functor.fancy}
    There exists a functor 
    \[\hat \Lie \colon \SingLocLieGrpd \to \Alg  \]
    with the following two properties:
    \begin{enumerate}[(a)]
        \item For all wide local Lie groupoid charts \( \pi \colon \til \CG \to \CG\) we have that \( \hat \Lie (\pi) \) is an isomorphism.
        \item \( \hat \Lie  = \Lie \) when restricted to the subcategory of local Lie groupoids.
    \end{enumerate}
    Furthermore, such a functor is unique up to a natural isomorphism.
 \end{theorem}
 \begin{proof}
    Recall that \( \Lie \colon \LocLieGrpd \to \Alg\) is an equivalence of categories.

    Therefore, we shall instead prove a closely related fact from which the above theorem will immediately follow:

    We claim that there exists a unique functor 
    \[\BF \colon \SingLocLieGrpd \to \LocLieGrpd  \]
    with the following two properties:
    \begin{enumerate}[(1)]
        \item For all wide local Lie groupoid charts \( \pi \colon \til \CG \to \CG\) we have that \( \BF[\pi] \) is an isomorphism.
        \item \( \BF|_{\LocLieGrpd} = \Id \).
    \end{enumerate}

    First let us construct such a \( \BF \).
    For each local singular Lie groupoid \( \CG \) chose a local groupoid chart \( \pi_\CG \colon \til \CG \to \CG \). 
    We make this choice in such a way so that when \( \CG \) is a local Lie groupoid we take \( \til \CG = \CG \). 
    Therefore, for each such \( \CG \) we define \( \BF(\CG) := \til \CG \).

    If \( [\CF] \colon \CG \to \CH \) is a morphism of singular local Lie groupoids, let \( \BF[\CF] \colon \til \CG \to \til \CH \) be the unique morphism which makes the following diagram commute:
    \[ 
    \begin{tikzcd}
        \BF(\CG) \arrow[d, "\pi_\CG"] \arrow[r, "{\BF[\CF]}"] & \BF(\CH) \arrow[d, "\pi_\CH"] \\
        \CG \arrow[r, "{[\CF]}"] & \CH
    \end{tikzcd}\]
    Such a morphism is guaranteed to exist by Theorem~\ref{theorem:existence.and.uniqueness.of.lifts}. 

    We must show that \( \BF \) is a functor. Suppose \( [\CF_1] \colon \CG \to \CH \) and \( [\CF_2] \colon \CH \to \CK \) are a pair of morphisms of singular local Lie groupoids. Then we get a commutative diagram in \( \SingLocLieGrpd \):
    \[ 
    \begin{tikzcd}[column sep = large]
        \BF(\CG) \arrow[d, "{[\pi_\CG]}"] \arrow[r, "{\BF[\CF_1]}"] & \BF(\CH) \arrow[d, "{[\pi_\CH]}"] \arrow[r, "{\BF[ \CF_2]}"] & \til \BF(\CK) \arrow[d, "{[\pi_\CK]}"] \\
        \CG \arrow[r, "{[\CF_1]}"] & \CH \arrow[r, "{[\CF_2]}"] & \CK
    \end{tikzcd}
    \]
    From the definition of \( \BF([\CF_2] \circ [\CF_1])\) we conclude from the uniqueness part of Theorem~\ref{theorem:existence.and.uniqueness.of.lifts} that:
    \[ \BF([\CF_2] \circ [\CF_1]) = \BF[\CF_2] \circ \BF[\CF_1]\]
    Now we will show that \( \BF \) satisfies (1) and (2) above.
    
    (1) Suppose \( \pi \colon \til \CG \to \CG \) is a local groupoid chart of a singular Lie groupoid. From the definition of \( \CF \) we have a commuting diagram:
    \[
     \begin{tikzcd}
        \BF(\CG) = \til \CG \arrow[d, "{[\Id_{\til \CG}]}"]  \arrow[r, "{\BF[\pi]}"] & \BF(\CG) \arrow[d, "\pi_\CG"]  \\
        \til \CG \arrow[r, "{[\pi]}"] & \CG
    \end{tikzcd}
    \]
    Since \( \pi \) is a quasi-étale chart, it follows from Proposition~\ref{proposition:local presentations.of.morphisms} that \( \BF[\pi] \) must be a submersion in a neighborhood of the identity elements. By a dimension count, we conclude that \( \BF[\pi]\) is a diffeomorphism in a neighborhood of the identity elements and so it is an isomorphism of local Lie groupoids.

    Property (2) is immediate from the fact that \( \BF(\CG) = \CG \) by definition for \( \CG \in \LocLieGrpd\).
    
    Now suppose \( \BF' \) is another functor satisfying properties (1) and (2) above. Given \( \CG \in \SingLocLieGrpd\) let:
    \[ \eta(\CG) := \BF' [\pi_\CG] \colon \BF (\CG) \to \BF'(\CG) \]
    The domain and codomain are as above since \( \BF' \) satisfies property (2). Furthermore, for each \( \CG\) we have that \( \eta(\CG) \) is an isomorphism since \( \BF' \) satisfies property (1).
    We only need to check that \( \eta \) defines a natural transformation. 
    
    Suppose we have \( [\CF] \colon \CG \to \CH\). From the definition of \( \BF[\CF] \) we have that the following square commutes:
    \[
    \begin{tikzcd}
        \BF (\CG) \arrow[r, "{\BF[\CF]}"] \arrow[d, "{[\pi_\CG]}"] & \BF(\CH) \arrow[d, "{[\pi_\CG]}"] \\
        \CG \arrow[r, "{[\CF]}"] & \CH
    \end{tikzcd}
    \]
    If we apply the functor \( \BF'\) to this square we get a new commuting square:
    \[
    \begin{tikzcd}
        \BF (\CG) \arrow[r, "{\BF[\CF]}"] \arrow[d, "\eta(\CG)"] & \BF(\CH) \arrow[d, "\eta(\CH)"] \\
        \BF' (\CG) \arrow[r, "{\BF'[\CF]}"] & \BF'(\CH)
    \end{tikzcd}
    \]
    Which proves that \( \eta\) is a natural transformation.
 \end{proof}
By viewing singular Lie groupoids as singular local Lie groupoids, it follows that Theorem~\ref{theorem:main.lie.functor} is a direct corollary of Theorem~\ref{theorem:main.lie.functor.fancy}.
 
 Although the definition of the Lie functor above is a bit abstract. Computing this functor is not too difficult for many kinds of singular Lie groupoids. Let us fix a choice of functor \( \hat \Lie \) satisfying Theorem~\ref{theorem:main.lie.functor.fancy}.

 Suppose \( \CG \) is a singular Lie groupoid and let \( \pi \colon \til \CG \to \CG \) be a wide local groupoid chart. Notice that if we apply the functor \( \hat \Lie \) to such a local groupoid chart we get:
 \[
     \hat \Lie [\pi] \colon \Lie (\til \CG) \to \hat \Lie(\CG)
 \]
 In other words, the Lie algebroid of the domain of any local groupoid chart is canonically isomorphic to the Lie algebroid of \( \til \CG\).
 Therefore, we can compute \( \hat \Lie \) by simply constructing a local groupoid chart and then applying the classical Lie functor.

 \begin{example}
     Suppose \( G \) is a Lie group and suppose \( N \subset G \) is a totally disconnected normal subgroup. We observed earlier that \( G / N \) is a singular Lie groupoid (over a point) and the projection map \( \pi \colon G \to G/N\) is a local groupoid chart. Therefore, \( \hat \Lie(G/N) \cong \hat \Lie(G) \).

     This leads to some interesting calculations. For example, the group \( \RR/\QQ\) has rather degenerate topology but \( \hat \Lie(\RR/ \QQ) \cong \RR \) so it has a perfectly acceptable Lie algebra.
 \end{example}
 \subsection{Example - singular Lie groupoids with integrable algebroids}
 This section is not necessary for proving our main theorems. However, it does tell us what a singular Lie groupoid with an integrable algebroid must look like. It generalizes the example of the singular Lie group.
 \begin{lemma}
    Suppose \( \til \CG \grpd M \) is a Lie groupoid and \( \CN \subset \til \CG \) is a wide normal subgroupoid. We think of \( \CN \) as a diffeological space via the subspace diffeology. Suppose further that \( \CN \) has the following properties:
    \begin{enumerate}[(a)]
        \item \( \CN \) includes only isotropy arrows. In other words, \( s|_\CN = t|_\CN\).
        \item The smooth map \( s|_\CN \colon \CN \to M \) is a local subduction.
        \item For each \( x \in M \) the fiber \( N_x := s\inv(x) \cap N \) is totally disconnected.
    \end{enumerate}
    Then \( \CG := \til \CG/\CN\) with the quotient diffeology is a singular Lie groupoid.
 \end{lemma}
 \begin{proof}
     First we show that the projection map \( \pi \colon \til \CG \to \CG \) is quasi-étale. The fibers of \( \pi \) are clearly totally disconnected due to property (c) above. From the definition of the quotient diffeology, we know \( \pi \) is a subduction but we must show it is a local subduction.

     Suppose \( \phi \colon U_\phi \to \CG \) is a plot and \( u_0 \in U_\phi \) and \( g_0 \in \til \CG \) are such that \( \phi(u_0) = \pi(g_0)\). 
     Since \( \pi \) is a subduction, we know that there must exist a smooth function \( \psi \colon V \to \til \CG \) such that \( V \subset U_\phi\) is an open neighborhood of \( u_0 \) and \( \pi \circ \psi = \phi\).

    Let \( n \in \CN \) be the unique element such that \( \psi(u) \cdot n = g \). 
    Now let \( \sigma \colon \CO \to \CN \) be a smooth section of \( s |_{\CN} \) such that \( \CO \) is an open neighborhood of \( s(n) \) and \( \sigma(s(n)) = n \). Such a section exists since \( s|_\CN\) is a local subduction.

    The function:
    \[ \til \phi(u) := \psi(u) \sigma(\s(u)) \]
    will be well defined in an open neighborhood of \( u_0 \) in \( U_\phi\). Furthermore, \( \pi \circ \til \phi = \phi\) and \( \til \phi(u_0) = g_0 \).

    Now we need to show that \( \pi \) is quasi-étale. We will use the simplified criteria from Lemma~\ref{lemma:quasi.étale.alternate}. Suppose \( f \colon \CO \to \til \CG \) is a smooth function defined on an open \( \CO \subset \CG \) such that \( \pi \circ f = \pi \) and \( f(g_0) = g_0 \) for some point \( g_0 \in \til \CG \). We need to show that \( f \) is a diffeomorphism in a neighborhood of \( g_0 \).

    Let us set some notation, to avoid confusion we will write \( \til s \) and \( \til t \) to denote the source and target maps for \( \til \CG \) and \( s \) and \( t \) to denote the source and target maps for \( \CG \).

    Assume without loss of generality that for all \( x \in M \) we have that \( \til\t\inv(x) \cap \CO \) is connected. Consider the function:
    \[ \alpha  \colon \CO \to \CN \qquad \alpha(g) := f(g) \cdot g\inv \]
    Since the the fibers of \( \CN \to M \) are totally disconnected, it follows that, for all \( x \in M \), the restriction of \( \alpha \) to  \( \til t\inv(x) \cap \CO \) is constant. In other words, \( \alpha \) is constant on target fibers. Therefore, there exists a unique function \( \sigma \colon \til \t(\CO) \to \CN \) which is a section of the source map and with the property that:
    \[ \forall g \in \CO \qquad \alpha(g) = \sigma(\til \t(g))   \]
    We can rewrite this to get that:
    \[ \forall g \in \CO \qquad f(g) \cdot g\inv = \sigma( \til \t(g)) \]
    In other words:
    \[ \forall g \in \CO \qquad f(g) = \sigma(\til \t(g)) \cdot g \]
    From this equation it follows that:
    \[ \forall g \in f(\CO) \qquad f\inv(g) = \sigma(\til t(g))\inv \cdot g \]
    Since \( f \) has a smooth inverse, \( f \) is a diffeomorphism onto its image. This shows that \( \pi \) is a quasi-étale chart.

    To finish the proof. We need to show that \( s \colon \CG \to M \) is a \( \QUED \)-submersion. By Theorem~\ref{theorem:qeds.submersion.criteria}, it suffices to show that the fibers of \( t \) are quasi-étale.

    Let us fix \( x \in M \) and consider the projection:
    \[ \pi_x \colon \til \t\inv(x) \to \t\inv(x) \]
    We claim that \( \pi_x \) is a quasi-étale chart. A standard argument shows that \( \pi_x \) is a local subduction. Of course, the fibers of \( \pi_x\) are totally disconnected.

    To show \( \pi_x \) is quasi-étale. We will use the simplified criteria from Lemma~\ref{lemma:quasi.étale.alternate}. Suppose \( f \colon \CO \to \til \t \inv(x)  \) is a smooth function defined on an open \( \CO \subset \til \t\inv(x) \) such that \( \pi_x \circ f = \pi_x \) and \( f(g_0) = g_0 \) for some point \( g_0 \in \til \CG \). We need to show that \( f \) is a diffeomorphism in a neighborhood of \( g_0 \).

    The argument that this \( f \) is a local diffeomorphism is essentially identical to the one from earlier on in this proof. We divide \( f \) by the identity map and observe that \( f \) must locally be left translation by an element of \( \CN \).
 \end{proof}
 The converse to the above lemma is also true. Every source connected singular Lie groupoid with an integrable Lie algebroid is the quotient of a Lie groupoid by a totally disconnected wide normal subgroupoid.
 \begin{lemma}
     Suppose \( \CG \grpd M \) is a source connected Weistein groupoid and \( \hat \Lie(\CG) \) is an integrable Lie algebroid. Then there exists a Lie groupoid \( \CG \) with a wide normal subgroupoid \( \CN \) satisfying properties (a), (b) and (c) from the previous lemma such that \( \CG \cong \til \CG / \CN \).  
 \end{lemma}
 \begin{proof}
     Let \( \pi^\circ \colon \til \CG^\circ \to \CG \) be a wide local groupoid chart. By a theorem of Fernandes and Michiels~\cite{fernandes_associativity_2020}, it is possible to chose \( \til \CG^\circ \) in such a way that \( \til \CG^\circ \) is an open subset of a source simply connected Lie groupoid \( \til \CG \).

     Using the associative completion functor of Fernandes and Michiels, one can extend the map \( \pi^\circ \) to a local groupoid chart \( \pi \colon \til \CG \to \CG \) defined on all of \( \til \CG \). Since \( \pi \) is open and \( \CG \) is source connected, it follows that \( \pi \) is surjective.

     Let \( \CN = \ker \pi \). We only need to show that \( \CN \) satisfies properties (a), (b) and (c). 
     
     Property (a) is immediate since \( \pi \) covers the identity map at the level of objects. 

     For property (b), suppose we have a plot \( \phi \colon U_\phi \to M \) and \( u  \in U_\phi \) together with \( n \in \CN \) such that \( s(n) = \phi(u)\).  Since the unit embedding \( \u \colon M \to \CG \) is smooth the map \( u \circ \phi \colon U_\phi \to \CG \) is a plot. Since \( \pi \colon \til \CG \to \CG \) is a local subduction, there must exist an open neighborhood \( V \subset U_\phi \) of \( u \) and a lift \( \til \phi \colon V \to \til \CG \) such that \( \pi \circ \til \phi = u \circ \phi|_{V} \) and \( \til \phi(u) = n \).

     Since \( \pi \circ \til \phi = u \circ \phi|_V \) it follows that \( s \circ \til \phi = \phi\) and the image of \( \til \phi \) is contained in \( \CN\).
 \end{proof}
\section{Proof of Theorem~\ref{theorem:existence.and.uniqueness.of.LGC} and Theorem~\ref{theorem:existence.and.uniqueness.of.lifts}}\label{section:proof.of.main.theorem}
Suppose \( \CG \grpd M \) is a local singular Lie groupoid. Following our usual convention, let us write \( s, t, u, m \) and \( i \) to denote the source, target, unit, multiplication and inverse groupoid structure maps for \( \CG \grpd M \), respectively. We will also use:
\[ \delta \colon \CD \to \CG \qquad (g,h) \mapsto \m(g,\i(h)) \]
to denote the division map. Note that the domain of division:
\[ \CD  := \{ (g,h) \in \CG \fiber{\s}{\s} \CG \ : \ \m(g,\i(h)) \text{ is well-defined} \} \]
is an open neighborhood of the image of \( \u \times u \colon M \to \CG \fiber{\s}{\s} \CG \).
It will be desirable for us to be able to compare elements of \( \CG \) by dividing them. Our next lemma states that any local singular Lie groupoid is isomorphic to one where this is the case.
\begin{lemma}
Suppose \( \CG' \grpd M \) is a local singular Lie groupoid.

Then \( \CG' \) is isomorphic (as a local groupoid) to a local singular Lie groupoid \( \CG \grpd M \) with the property:
\[ \forall (g,h) \in \CD \qquad \delta(g,h) = \u \circ \t (g) \quad \Leftrightarrow \quad g = h \} \]
\end{lemma}
\begin{proof}
Consider the following calculation:
\begin{align*}
   (gh\inv)h &= (\u \circ \t(g))h \\
   g (h\inv h) &= h \\
   g (\u \circ \s(h)) &= h \\
   g &= h
\end{align*}
Let \( \CG \subset \CG' \) be an open neighborhood of the units with the property that for all \( g, h \in \CG \) we have that each step of the above calculation is well-defined. Since being well-defined is an open condition (it is just about being in the inverse image of open sets under some continuous functions) it follows that \( \CG \) is an open set. \( \CG \) will be an open neighborhood of the units since the above calculation is always well-defined for units.

Now if we have that \( \delta(g, h) = \u \circ \t(g) \) for \( g,h \in \CG \), it follows from the above calculation that \( g = h \).
\end{proof}
\begin{remark}
The above proof can be generalized into the following principle: For any local groupoid and a finite number of equations that are consequences of the groupoid axioms, there exists an open local subgroupoid where the desired equation holds. Crucially, this only holds for a \emph{finite} number of equations as an infinite intersection of open sets may not be open.
\end{remark}

\subsection{Lifting the division map}

We will begin by showing that, given a quasi-étale chart \( \pi \colon \til \CG \to \CG \), we have can lift the division operation to \( \til \CG \) in a way that has favorable properties.
 
Given a quasi-étale chart \( \pi \colon \til \CG \to \CG \) we will use the conventions that: \( \til s := s \circ \pi \) and \( \til t  := t \circ \pi \). 
Since \( s \), \( t \) and \( \pi \) are local subductions, \( \til s \) and \( \til t \) will be submersions. 

Rather than choosing local representations of the normal groupoid structure maps, we will begin by choosing a local presentation of division. 
This is due to the fact that all of the remaining structure maps can be recovered from division. 
Indeed, (local) groupoids can be studied entirely in terms of their division map and source map (see the appendix of Crainic, Nuno Mestre, and Struchiner~\cite{crainic_deformations_2020}.

Our first lemma says that it is possible to find a lift of the division map:
\begin{lemma}\label{lemma:delta.exists}
Let \( \CG \) be a singular Lie groupoid and let us fix a point \( x_0 \in M \). Suppose \( \pi \colon \til \CG \to \CG \) is a quasi-étale chart and we have a point \( e \in \til \CG \) such that \( \pi(e) = \u(x_0)\).

There exists an open neighborhood \( \til \CD \subset \til \CG \fiber{\til s}{\til \t} \til \CG \) of \( (e,e) \) together with a submersion \( \delta \colon \til \CD \to \til \CG \) such that the following diagram commutes:
\begin{equation}\label{diagram:delta.is.lift}
\begin{tikzcd}[column sep=huge]
\til \CD \arrow[d, "(\pi \circ \pr_2) \times (\pi \circ pr_2)", swap] \arrow[r, "\til \divi"] & \til \CG \arrow[d, "\pi"] \\
\CD  \arrow[r, "\divi"] & \CG
\end{tikzcd}
\end{equation}
\end{lemma}
\begin{proof}
Note that since \( \s \) is a \( \QUED\)-submersion, it follows that \( \CG \fiber{s}{s} \CG \) is quasi-étale and the map:
\[ \pi \circ \pr_1 \times \pi \circ \pr_2 \colon \til \CG \fiber{s}{s} \til \CG  \to  \CG \fiber{s}{s} \CG\]
is a quasi-étale chart.

Since the domain of division, \( \CD \), is an open subset of a quasi-étale space it follows that \( \CD \) is quasi-étale as well.

Since \( \pi \) is a local subduction, is possible to chose a local representation of the division map \( \til \delta \colon \til \CD \to \til \CG \) where \( \til \CD \subset \til \CG \fiber{\s}{\s} \til \CG \) is an open neighborhood of \( (e,e) \) and which makes Diagram~\ref{diagram:delta.is.lift} commute.

Furthermore, we can apply Lemma~\ref{lemma:fiber.hopping}, to chose \( \til \delta \) in such a way that \( \til \delta(e,e) = e \).

Since \( \divi \) is a local subduction, Proposition~\ref{proposition:local presentations.of.morphisms} tells us that \( \til \divi \) will be a submersion. 
\end{proof}
The core of our proof of the existence of local groupoid charts will be the fact that a map \( \delta \) as above will always be (in some open neighborhood) the division map for a local groupoid.

Our next lemma tells us that a choice of representation of the division map also induces a lift of the unit embedding.

\begin{lemma}\label{lemma:fix.units}
Let \( \CG \) be a singular Lie groupoid and let us fix a point \( x_0 \in M \). Suppose \( \pi \colon \til \CG \to \CG \) is a quasi-étale chart and we have a point \( e \in \til \CG \) such that \( \pi(e) = \u(x_0)\). Let \( \delta \) be a representation of division as in Lemma~\ref{lemma:delta.exists}.

There exists a smooth function 
\[ \til \u \colon \til M \to \til \CG \] 
where \( \til M \) is an open neighborhood of \( x_0 \) in \( M \) and such that:
\begin{itemize}
\item \( \pi \circ \til \u = \u \)
\item \( \til \u(x_0) = e  \)
\item \( \forall x \in \til M  \) we have that \( \til \divi(\til \u(x), \til \u(x)) = \til \u(x) \).
\end{itemize}
\end{lemma}
\begin{proof}
First let \( \CU \subset \til \CG \) be an open neighborhood of \( e \) such that for all \(g \in \CU \), we have that \( (g,g) \in \til \CD \).
In other words, \( (g,g) \) is in the domain of \( \til \delta\).
Write \( \Delta \colon \CU \to \til \CD \) to denote the diagonal embedding.

Now, note that for all \( g \in \CU \), compatibility of \( \divi \) with \( \pi \) implies that: 
\[ \til \t ( g)  = (\til \s \circ \til \divi \circ \Delta)(g) \] 
Since \( \til t \) is a surjective submersion, this implies that \(\til \divi \circ \Delta \)
has rank at least equal the dimension of the object manifold \( M \).

On the other hand, if \( g(t) \) is a curve in \(  \CU \) tangent to a fiber of \( \til \t|_{\CU} \), then it follows that \( (\pi \circ \til \divi \circ \Delta)(g(t)) \) is a constant path. 
Since \( \pi \) is quasi-étale, it follows that \( g(t) \) is a constant path in a fiber of \( \til \divi \circ \Delta \).
This implies that kernel distribution of \( \til \t|_{\CU} \) is contained in the kernel distribution of \( \til \divi \circ \Delta \). 
In other words: 
\[ \ker T \til \t \subset \ker T (\til \divi \circ \Delta) \]
Since the rank of \( \til \divi \circ \Delta \) is at least equal to the rank of \( \til \t |_{\CU}\), a dimension count tells us that the kernel distributions are actually equal.

Now let us shrink \( \CU \) to a smaller neighborhood of \( e \in \til \CG \) with the property that the fibers of \( \til \t|_{\CU} \colon \CU \to M \) and \( \til \delta \circ \Delta \colon \CU \to \til \CG \) coincide. 
This means that for all \( g, h \in \CU \), we have that:
\begin{equation}\label{eqn:fix.units}
 \til \t  (g) = \til \t(h) \quad \Leftrightarrow \quad \til \divi(g,g) = \til \divi(h,h) 
 \end{equation}
From all of these facts, we conclude that \( \til t \) is a diffeomorphism when restricted to the image of \( \til \divi \circ \Delta \). In other words, the image of \( \til \divi \circ \Delta \) must be the image of a section \( \til \u \colon \til M \to \til \CG \) of \( \til \t \).

Compatibility of \( \pi \) with \(\til \divi \) implies that \( \pi \circ \til \u = \u \). 
Furthermore, since \( \til \divi( e , e ) = e \) we know that \( \til \u(x_0) = e  \). 
Finally, for any \( x \in \til M \), we have that:
\[ \til \t (\til u(x) ) = \til \t ( \til \divi( \til \u(x), \til \u(x))) \]
By (\ref{eqn:fix.units}) we conclude that \( \til \divi(\til \u(y), \til \u(y)) = \til \u(y) \).
\end{proof}
The final lemma for this subsection tells us that a lift of the division map (as in the above lemmas) can be used as an equality test:
\begin{lemma}\label{lemma:division.lemma}
Let \( \CG \) be a local singular Lie groupoid and let us fix a point \( x_0 \in M \). Suppose \( \pi \colon \til \CG \to \CG \) is a quasi-étale chart and we have a point \( e \in \til \CG \) such that \( \pi(e) = \u(x_0)\). Let \( \delta \colon \til \CD \to \til \CG \) be a representation of division as in Lemma~\ref{lemma:delta.exists} and \( \til \u \colon \til M \to \til \CG \) be a lift of the units as in Lemma~\ref{lemma:fix.units}.

There exists an open neighborhood \( \CO  \subset \til \CG \) of \( e\) with the following properties:
\begin{itemize}
\item \( \CO \times_{\til \s, \til \s} \CO \subset \CD \)
\item For all \( (g, h) \in \CO \times_{\til \s, \til \s} \CO \) we have that \( g = h \) if and only if \( \til \delta (g,h) = \til u({\til \t(g)})  \). 
\end{itemize}
\end{lemma}
\begin{proof}
First observe that \( \til u(\til M) \) is a cross section of a submersion and is therefore an embedded submanifold. 
Since \( \til \divi \) is a submersion, we know that the inverse image \( \til \divi^{-1} (\til u(\til M)) \) is an embedded submanifold of dimension equal to the dimension of \( \til \CG \). 

Now let \( \CU \subset \til \CG \) and \( \Delta \colon \CU \to \CD \) be as in the proof of Lemma~\ref{lemma:fix.units}.

From the construction of \( \til \u \) we know that:
\[\Delta (\CU) \subset \til \divi^{-1}(\til u(\til M) )\]
By a dimension count, \( \Delta (\CU) \) is actually an open subset of  \( \til \divi^{-1}(\til u (\til M )  )  \) containing \( (e,e) \). 
Now let \( \CO \subset \til \CG \) be an open neighborhood of \( e  \) such that:
\[  (\CO \times_{\til s,\til s} \CO) \cap \til \divi^{-1}(\til u( \til M ))) \subset  \Delta(\CU) \]
We claim that this is the desired open subset. 
Suppose \( g, h \in \CO \) have the same source and \( \til \divi(g,h) = \til u({\til \t(g)}) \). 
Then:
\[ (g,h) \in   (\CO \times_{s,s} \CO) \cap \til \divi^{-1}(\til u(\til M))  \]
Therefore \( (g,h) \in \Delta(\CU ) \) and \( g = h \).
\end{proof}

\subsection{Division structures and comparing maps}
It will be useful to formalize the properties from the previous three lemmas into a definition.
\begin{definition}
     Let \( \CG \grpd M \) be a local singular Lie groupoid.
     Suppose \( \pi \colon \til \CG \to \CG \) is a quasi-étale chart. A \emph{division structure} on \( \pi \) consists of the following:
     \begin{itemize}
     \item A smooth function: 
     \[ \til \u \colon  \til M \to \til \CG \]
     where \( \til M = \til \s(\til \CG) \subset M \).
     \item A smooth function:
     \[
     \til \delta \colon \til \CD \to \til \CG
     \]
     where \( \til \CD \subset \til \CG \fiber{\til s}{\til \s} \til \CG \) is an open neighborhood of the image of \( \til \u \times \til \u\).
     \end{itemize}
     We require these two functions to satisfy the following properties:
     \begin{enumerate}[(a)]
         \item \( \til \delta \) is a representation of division. In other words we have a commutative diagram:
         \begin{tikzcd}[column sep=huge]
            \til \CD \arrow[d, "(\pi \circ \pr_2) \times (\pi \circ pr_2)", swap] \arrow[r, "\til \divi"] & \til \CG \arrow[d, "\pi"] \\
            \CD  \arrow[r, "\divi"] & \CG
        \end{tikzcd}
        \item \( \til \u \) is a lift of the units. In other words, \( \pi \circ \til \u = \u |_{\til M} \).
        \item For all \( (g,h) \in \til \CD \) we have that:
        \[ \til \delta(g,h) = \til u \circ \til \t(g) \quad \leftrightarrow \quad  g = h \]
     \end{enumerate}
 \end{definition}
 The sum effect of Lemma~\ref{lemma:delta.exists}, Lemma~\ref{lemma:fix.units} and Lemma~\ref{lemma:division.lemma} is the claim that around any point \( x_0 \in M \) we can find a quasi-étale chart \( \til \pi \colon \til \CG \to \CG \) around \( \u(x) \) equipped with a division structure.

 The power of division structures is that they enable us to develop very useful tests for equality of certain maps. First let us establish some notation. Given a natural number \(n \) and \( U \subset \til \CG \) open, we will write: 
\[U^{(n)} := \overbrace{U \times_{\til s,\til t} U \times_{\til s,\til t} \cdots \times_{\til s,\til t} U}^{n\text{-times}} \]
In other words, \( U^{(n)} \) is the set of \( n \)-tuples of ``composable'' elements of \( U \).

 Our next lemma provides us with a way of determining when exactly a function on the composable arrows takes values only in units. 
\begin{lemma}\label{lemma:failure.function}
    Suppose \( \CG \grpd M \) and \( \CH \grpd N \) are a local singular Lie groupoid and \( \pi \colon \til \CG \to \CG \) is a quasi-étale chart equipped with a division structure.

    Suppose \(  \CU \subset \til \CG \) is an open neighborhood of the image of \( \til \u \). Given a natural number \( n \) suppose that we have a smooth function:
    \[ F \colon \CU^{(n)} \to \til \CG \]
    with the following properties:
    \begin{itemize}
        \item The image of \( \pi \circ F\) contains only unit arrows in \( \CG \).

        \item \( \til t \circ F = \til t \circ \pr_1 \)
        \item For all \( x \in \til u\inv(\CU) \), we have that \( F(\til u(x), \ldots , \til u(x)) = \til u(x) \)
    \end{itemize}
    Then there exists an open neighborhood \( \CO \subset \CU  \) of the image of \( \til \u \) with the property that for all \( (g_1, \ldots , g_n ) \in \CO^{(n)}\) we have that
    \[F(g_1, \ldots , g_n) = (\,\til u \circ \til t \,)(g_1)  \]
    In particular, for a small enough open neighborhood of \( \CU \), the function \( F \) takes values only in the image of \( \til u \).
\end{lemma}
\begin{proof}
    Notice that since \( \til t \circ \pr_1 \) is a submersion and
\[ \til \t \circ F = \til \t \circ \pr_1 \]
this implies that \( F \) has rank at least equal to the rank of \( \til \t \circ \pr_1 \) which is equal to the dimension of \( M \). Now we claim that the kernel distribution of \( TF \) contains the kernel distribution of \( T (\til t \circ \pr_1) \). By a dimension count, this will imply their kernel distributions are equal.

To see this, suppose \( \gamma \) is a path tangent to the kernel distribution of \( T\til t \circ \pr_1 \).
We know that \( \til t \circ F \circ \pi \) is constant.
Furthermore, since \(\pi \circ  F \) contains only unit elements, we can conclude that \( F \circ \gamma \) is a path in \( \pi\)-fiber of a unit element in \( \CG \).
However, the fibers of \( \pi \) are totally disconnected, and so this implies that \( F \circ \gamma \) is a constant path.
Therefore, \( TF \) and \( T(\til t \circ \pr_1) \) are submersions with identical kernel distributions

Now, notice that the map:
\[ \til \u^{(n)} \colon \til M \to \til \CG^{(n)} \qquad x \mapsto ( \til u(x) , \til u(x) , \ldots , \til u(x) ) \]
is a section of \( \til t \circ \pr_1\).
By the local normal form theorem for submersions around a section, it is possible to find an open neighborhood \( \CW \) of the image of \( \til \u^{(n)} \) of the domain of \( F \) with the property that the fibers of \( \til t \circ \pr_1 \) and \( F \) are connected and coincide.

Now we claim that 
\[\til F|_{\CW} = \til u \circ \til t \circ \pr_1 |_{\CW} \]
To see why, first notice that the fibers of these two functions coincide. Furthermore 
\[ F(\til u(x), \ldots , \til u(x)) = \til \u(x) =  \til u \circ \til t \circ \pr_1(\til u(x), \ldots , \til u(x)) \]
Since the fibers of \( \til F|_{\CW}\) and \( \til u \circ \til t \circ \pr_1 |_{\CW}\) are equal and they agree on one element of each fiber, the two functions are equal.

Therefore, the proof is completed by choosing an open neighborhood \( \CO \subset \CU \) of the units with the property that \( \CO^{(n)} \subset \CW \).
\end{proof}
Our next lemma is just an upgrade of the previous lemma.
It provides us with a very useful equality test.
\begin{lemma}\label{lemma:equality.test}
   Suppose \( \CG \grpd M \) is a local singular Lie groupoid and \( \pi \colon \til \CG \to \CG \) is a quasi-étale chart equipped with a division structure.

    Suppose \( \CU \subset \til \CG \) is an open neighborhood of the image of \( \til \u \). Given a natural number \( n \), suppose that we have a pair of smooth functions:
    \[ \alpha \colon \CU^{(n)} \to \til \CG \qquad \beta \colon \CU^{(n)} \to \til \CG\]
    with the following properties:
    \begin{itemize}
        \item \( \pi \circ \alpha = \pi \circ \beta \)
        \item  \( \til t \circ \alpha = \til t \circ \pr_1\)
        \item For all \( x \in \til u\inv(\CU) \), we have that  
        \[ \alpha(\til u(x), \ldots , \til u(x)) = \beta(\til u(x), \ldots , \til u(x)) = \til u(x) \]
    \end{itemize}
    Then there exists an open neighborhood \( \CO \subset \CU \) of the image of \( \til \u\) such that \( \alpha|_{\CO^{(n)}} = \beta|_{\CO^{(n)}} \)
\end{lemma}
\begin{proof}
    Let:
    \[ F \colon \CU^{(n)} \to \til \CG  \qquad F(g_1,\ldots , g_n) := \til \delta(\alpha(g_1, \ldots , g_n), \beta(g_1,\ldots , g_n)) \]
    We claim that \( F\) satisfies the hypotheses of Lemma~\ref{lemma:failure.function}.
    For the first bullet point, note that for \( \overline{g} \in \CU^{(n)}\):
    \[ \pi \circ F (\overline{g}) = \delta(\pi \circ \alpha(\overline{g}), \pi \circ \beta ({\overline{g}})) \]
    Since we have assume that \( \pi \circ \alpha = \pi \circ \beta\) it follows that dividing them in \( \til \CG \) results in a unit element. Therefore \( \pi \circ F \) takes values only in unit elements.
    For the second bullet point, note that since \( \til t \circ \til \delta = \pr_1 \) it follows that \( \til t \circ F = \til t \circ \alpha \).
    By assumption, \( \til t \circ \alpha = \til t \circ \pr_1 \) so the bullet point holds.
    Finally, given \( x \in \til M \), we have that:
    \begin{align*}
    F(\til u(x), \ldots , \til u(x)) &= \til \delta( \alpha(\til u(x), \ldots , \til u(x)), \beta(\til u(x), \ldots , \til u(x))) \\
    &= \til \delta( \til u(x), \til u (x)) \\
    &= \til u(x)
    \end{align*}
    Since \( F \) satisfies the hypotheses of Lemma~\ref{lemma:failure.function}, it follows that there exists an open neighborhood \( \CO \subset \CU \) of \( e \) with the property that:
    \[ \forall (g_1, \ldots , g_n ) \in \CO^{(n)} \qquad  \til \delta(\alpha(g_1, \ldots , g_n), \beta(g_1, \ldots , g_n)) = \til u( \til t(g_1))  \]
    By Lemma~\ref{lemma:division.lemma}, we conclude that:
    \[\forall (g_1, \ldots , g_n ) \in \CO^{(n)} \qquad  \alpha(g_1, \ldots , g_n) = \beta(g_1, \ldots , g_n) \]
\end{proof}
\subsection{Local Existence of charts}
\begin{proposition}[Existence of local groupoid charts]\label{prop:existence.of.local.local.groupoid.charts}
Suppose \( \CG \grpd M \) is a singular Lie groupoid and let \( x_0 \in M \) be fixed. There exists a local groupoid chart \( \pi \colon \til \CG \to \CG  \) around \( x_0 \).
\end{proposition}
\begin{proof}
Let \( \pi \colon \til \CG' \to \CG \) be a quasi-étale chart equipped with a division structure around \( x_0 \). We saw earlier that such things exist.

Now let:
\[   \til \i \colon \CI \to \til \CG \qquad g \mapsto \divi(\til \u(\s(g)), g) \]
\[ \til \m \colon \CM \to \til \CG' \qquad (g,h) \mapsto \divi (g , \til \i (h)) \]
where \( \CI \subset \til \CG' \) and \( \CM \subset \til \CG' \times_{\til \s, \til t} \til \CG' \) are the maximal open sets which make these well-defined.

We claim that there exists an open neighborhood \( \til \CG \) of the image of \( \til \u \) which makes the above maps a local groupoid structure on \( \til \CG \) and \( \pi|_{\CG} \) is a local groupoid chart.

We will begin by proving that \( \pi \) is compatible with these structure maps.
\begin{enumerate}[(1)]
    \item (Compatibility with source and target) By definition \( \til \s = \s \circ \pi \) and \( \til \t = \t \circ \pi  \) so this condition is automatic.
    \item (Compatibility with multiplication) Consider the following computation where we apply compatibility of \( \til \delta\) with \( \pi \) multiple times:
    \begin{align*} 
    \pi \circ \til \m (g,h) &= \pi \circ \til \delta(g , \til \delta( \til \u \circ \til \s (h), h) ) \\
    &= \delta ( \pi(g), \pi \circ \til \delta( \til \u \circ \til \s(h), h) ) \\
    &= \delta(\pi(g), \delta( \pi \circ \til \u \circ \til \s (h), \pi(h))) \\
    &= \delta(\pi(g), \delta( \u \circ \s(\pi(h), \pi(h)))
    \end{align*}
    Since \( \CG \) is a local groupoid, there exists an open neighborhood of the units where 
    \[\delta(\pi(g), \delta( \u \circ \s(\pi(h), \pi(h))) = \m (g,h) \]
\end{enumerate}
A similar calculation to what we have done above tells us that we can find an open neighborhood of the units \( \til \CG \) where:
\[ \pi ( \til \i(g)) = \i( \pi(g))  \]
We can therefore assume without loss of generality that the whole ambient space \( \til \CG' \) has the property \( \til  \m \) and \( \til  \i \) are compatible with \( \pi\).

Now we will show the axioms (LG1-6) of a local groupoid must each hold in an open neighborhood of the image of \( \til \u \).
\begin{enumerate}[(LG1)]
    \item (Compatibility of source and target with unit) We will do the computation for the source as the proof of target is symmetrical:
    \[ \til \s \circ \til \u = s \circ \pi \circ \til \u = \s \circ \u = \u \]
    \item (Compatibility of source and target with multiplication) We will do the computation for the source as the proof of target is symmetrical:
   \[ \til \s \circ \til \m (g,h) = \s \circ \pi \circ \til \m(g,h) = \s \circ \m(\pi(g),\pi(h)) = \s \circ \pi(h) = \til \s(h)  \]
   \item (Compatibility of source and target with inverse)
   \[ \til s \circ \til \i(g) = \s \circ \pi \circ \til \i(g) = \s \circ \i \circ \pi(g) = t \circ \pi(g) = \til \t(g) \]
   \item (Left and Right unit law) We will show the left unit.
   First, observe that one can choose an open neighborhood \( \CU \subset \til \CG \) of the image of \( \til \u \) such that \( (\til u(\til \t(g)) , g) \in \CM \) for all \( g \in \CO \).

Given such an \( \CU \), consider the two functions:
\[ \alpha \colon \CU \to \til \CG \qquad  g \mapsto  \til \m( \til \u(\til \t(g)),g)  \]
\[ \beta \colon \CU \to \til \CG \qquad  g \mapsto   g   \]
Since \( \pi \) preserves our structure maps, it follows that \( \pi \circ \alpha = \pi \circ \beta\).
Furthermore, \( \til t \circ \alpha = \til \t \).
Lastly, if \( x \in \til M \) we have that:
\[ \alpha( \til u(x)) = \til \m ( \til \u(x) , \til u(x)) = \til u(x)\]
Therefore, the pair \( \alpha\) and \( \beta\) satisfy the hypotheses of Lemma~\ref{lemma:equality.test} (in the case where \( n = 1 \)) which completes the proof.
\item (Left and right inverse laws) We will show the proof for right inverse. The left inverse case is symmetrical.

As we did in the proof of left unit. Consider a pair of functions:
\[ \alpha \colon \CU \to \til \CG \qquad g \mapsto \til \m(g, \til \i(g) ) \]
\[ \beta \colon \CU \to \til \CG \qquad g \mapsto \til \u (\til \t(g)))\]
where \( \CU \) is an open neighborhood of the image of \( \til \u \) that makes \( \alpha \) and \( \beta \) well-defined.

Observe that since \( \pi \) is compatible with our structure maps and \( \CG \) satisfies the unit axiom, it follows that \( \pi \circ \alpha = \pi \circ \beta \).
Furthermore, \( t \circ \alpha = \til t \). 
Finally, if \(x \in \til M \) we have that:
\[ \alpha(\til u(x)) = \til \m ( \til \u (x) , \til \i (\til \u(x))) = \til \m (\til \u(x), \til \u(x)) = \til \u(x) = \beta( \til u(x)) \]
Therefore, \( \alpha\) and \( \beta\) satisfy the hypotheses of Lemma~\ref{lemma:equality.test} and it follows that \( \alpha = \beta \) in an open neighborhood of \( e \).
\item (Associativity Law) As we did for the previous two axioms. We consider a pair of functions:
\[ \alpha \colon \CO^{(3)} \to \til \CG \qquad (g,h,k) \mapsto \til \m (g, \til \m(h,k)) \]
\[\beta \colon \CO^{(3)} \to \til \CG \qquad (g,h,k) \mapsto \til \m( \til \m(g,h),k) ) \]
where \( \CU \) is an open neighborhood of the image of \( \til \u \) and is chosen in such a way that \( \alpha \) and \( \beta \) are well-defined.

Since \( \pi \) preserves these structure maps, and \( \CG \) satisfies associativity, it follows that \( \pi \circ \alpha = \pi \circ \beta\).
Furthermore, \( \til t \circ \alpha = \til t \circ \pr_1 \).
Finally, given \( x \in \til M \) we have that:
\[ \alpha( \til u (x),\til u (x),\til u (x) ) = \til m(\til u (x), \til m(\til u (x),\til u (x))) = \til m(\til u(x), \til u(x)) = \til u(x) \]
A similar calculation holds for \( \beta \).

Therefore, it follows that \( \alpha \) and \( \beta \) satisfy the hypotheses of Lemma~\ref{lemma:equality.test}, and therefore \( \alpha = \beta\) in some open neighborhood of the image of \( \til \u\).
\end{enumerate}
\end{proof}
 At this point, we have proved a local (non-wide) form of the existence part of the proof of Theorem~\ref{theorem:existence.and.uniqueness.of.LGC}. That is, every singular Lie groupoid admits a local groupoid chart around any given object.

 \subsection{Uniqueness of charts}
Before we explain why ``wide'' local groupoid charts exist. We first need to prove a local form of the uniqueness portion of Theorem~\ref{theorem:existence.and.uniqueness.of.LGC}.
 
First, we will observe that the local groupoid structure on a local groupoid chart is uniquely (in a neighborhood of units) determined by its unit embedding.
\begin{lemma}\label{lemma:unique.lift}
Let \( \CG \grpd M \) be a singular Lie groupoid. Suppose \( \pi \colon \til \CG \to \CG \) is a quasi-étale chart we have two different local groupoid structures on \( \til \CG \) which both make \( \pi \) into a local groupoid chart and which have the same set of units \( \til M \subset M \). 

Let  
\[\til \u \colon \til M \to \til \CG  \qquad  \til \u '\colon \til M \to \CG \] 
denote the two different unit embeddings.
Then for all \( x_0 \in \til M \) the local groupoid structures on \( \til \CG \) are equal in a neighborhood of \( \til u(x_0) \) if and only if \( \til \u \) and \( \til \u' \) are equal in a neighborhood of \( x \).
\end{lemma}
\begin{proof}
Let \( \til \m \) and \( \til \m ' \) denote the two different multiplication maps and let \( x \in \til M \) be fixed.
Note that one direction is clear, if the two local groupoid structures are equal then they must have the same unit embedding.
Therefore, we only need to show that if the unit embeddings are equal then \( \m \) and \( m' \) are equal in an open neighborhood of \( e := \til u(x_0) = \til u'(x_0) \).

By assumption, \( \pi \circ \til m = \pi \circ \til m '\). Furthermore, \( \til t \circ \til m = \til t \circ \pr_1 \). Finally, observe that for all \( x \in \til M \), we have that:
\[ \til m(\til u(x), \til u(x)) = \til m' (\til u(x), \til u(x)) \]
Therefore, the functions \( \alpha = \til \m\) and \( \beta = \til \m'\) satisfy the hypotheses of Lemma~\ref{lemma:equality.test} so they must be equal.
\end{proof}
Our next lemma tells us that we can always construct an isomorphism between any two local groupoid charts. 
\begin{lemma}\label{lemma:isomorphism.exists}
Suppose \( \pi \colon \til \CG \to \CG \) and \( \pi' \colon \til \CG' \to \CG \) are local groupoid charts around \( x_0 \in M \). Let \( \til \u \) and \( \til \u' \) be the respective unit embeddings. Then there exists an open neighborhood \( \CU \) of \( x_0\) and an isomorphism of local groupoids \( [F] \colon \til \CG|_{\CU} \to \til \CG'|_{\CU} \) such that \( [\pi] = [\pi'] \circ [F] \).
\end{lemma}
\begin{proof}
Without loss of generality, we can assume that the two local groupoid charts have the same set of units \( \til M \). 

Lemma~\ref{lemma:unique.lift} tells us that if two (local) groupoid structures on \(\til \CG' \) are compatible with \( \pi' \) and have the same set units near \( \til u'(x_0) \), then they are equal in a neighborhood of \( \til u'(x_0) \). 
Therefore, it suffices to construct a diffeomorphism \( F_1 \colon \til \CG \to \til \CG' \), defined in a neighborhood of \( \til u(x_0) \) such that \( F \circ \til \u = \til \u' \) near \( \til u( x_0) \). 
Such a diffeomorphism will automatically be compatible with the multiplication operations in some neighborhood of \( \til u(x_0) \).

Since \( \pi \) and \( \pi' \) are quasi-étale, we can assume without loss of generality that we have a diffeomorphism \( \hat F \colon \til \CG \to \til \CG' \) such that \( \hat F (\til \u(x)) = \til \u'(x) \). Let \( \til \delta' \) be the division map for \( \til \CG' \). Now consider the function:
\[ F(g) := \til \divi'( \hat F(g), \hat F ( \til \u(\til \s(g)) ) )  \] 
where \( \til \s \) is the source map for \( \til \CG \). 
Clearly \( F \) will be well-defined in a neighborhood of \( \til \u(x) \). 
A direct calculation shows that \( F \circ \til \u = \til \u' \) and \( \pi' \circ F = \pi\).
\end{proof}
We will finish the proof of this section a lemma where we show that there exists exactly one isomorphism between any two given local groupoid charts.
\begin{lemma}\label{lemma:isomorphism.is.unique}
Suppose \( \pi \colon \til \CG \to \CG \) and \( \pi' \colon \til \CG ' \to \CG  \) are local groupoid charts around \( x_0 \in M \). If \( \CU \subset M \) is an open neighborhood of \( x_0 \) and  \( [F] \colon\til \CG|_{\CU} \to \til \CG' |_{\CU} \) and \( [G] \colon\til \CG|_{\CU} \to \til \CG' |_{\CU}\) satisfy 
\[ [\pi] = [\pi']  \circ [F] \qquad [\pi] = [\pi']  \circ [G] \]
Then \( [F] = [G] \).
\end{lemma}
\begin{proof}
We showed in the previous lemma that such an isomorphism exists. Therefore, we need to show that it is unique.

We can assume without loss of generality that \( \CU = M \). We must show that \([F] \circ [G]\inv \colon \til \CG|_{\CU} \to [\til \CG]  \) is equal to the identity germ. 

Note that \( [F]\circ [G]\inv \) satisfies  \( [\pi]_{x_0} = [\pi]_{x_0} \circ( [F]_{x_0} \circ [G]_{x_0}) \) so we can reduce to the case where \( \til \CG' = \til \CG\).

Therefore, we are in a situation where we have a local groupoid map \( [F] \colon \til \CG \to \til \CG \) which is an isomorphism in a neighborhood of the units and is such that \( \pi \circ F = \pi \). 
We must show that \( F \) equals the identity.

Note that we have that \( \pi \circ F = \pi \circ \Id_{\til \CG} \), and also \( \til t \circ F = \til t\). Furthermore, if \( x \in \til M \) we have that \( \til F(\til u(x)) = \til u(x) \). Therefore, \( F\) and \( \Id_{\til \CG}\) satisfy the hypotheses of Lemma~\ref{lemma:equality.test} and therefore \( F = \Id \) in a neighborhood of \( \til u(x_0)\).
\end{proof}

\subsection{Proof of Theorem~\ref{theorem:existence.and.uniqueness.of.LGC}}
\begin{proof}
Suppose \( \CG \grpd M \) is a singular Lie groupoid. 

First, we show show uniqueness. Suppose \( \pi \colon \til \CG \to \CG \) and \( \pi' \colon \til \CG' \to \CG \) are wide local groupoid charts. By Lemma~\ref{lemma:isomorphism.exists}, around each point \( p \in M \) there exists an open neighborhood \( \CU_p \subset M\) of \( p \) and an isomorphism \( [\CF_p] \colon \til \CG|_{\CU_p} \to \til \CG'|_{\CU_p}\) with the property that \( [\pi'] \circ [\CF_p] = [\pi] \).

By Lemma~\ref{lemma:isomorphism.is.unique}, for any pair \( p, q \in M \) such that \( \CU_p \cap \CU_q\) is non-empty, it must be the case that 
\[ [\CF_p]|_{\CG|_{\CU_p \cap \CU_q}} = [\CF_q]|_{\CG'_{\CU_p \cap \CU_q}}    \]
Recall that the category of local groupoids is equivalent to the category of Lie algebroids. Since morphisms of Lie algebroids form a sheaf (over the base manifold), it follows that morphisms of local groupoids also form a sheaf. Therefore, there must exist a unique local groupoid morphism \( [\CF] \colon \til \CG \to \til \CG' \) with the property that \( [\pi'] \circ [\CF] = [\pi] \).

Now we consider existence. By Lemma~\ref{prop:existence.of.local.local.groupoid.charts}, we know that for all \( p \in M \) there must exist an open neighborhood \( \til M_p \subset M\) of \( p \) together with a local groupoid \( \til \CG_p \grpd \til M_p\) and a local groupoid chart \( \pi \colon \til \CG_p \to \CG\).

Furthermore, we remark that given any two \( p, q \in M \) such that \( \til M_p \cap \til M_q \neq \emptyset \), by the uniqueness part of the proof there must exist a unique isomorphism of local groupoid charts:
\[ [\CF_{pq}] \colon \til \CG_p |_{\CU_p \cap \CU_q} \to \til \CG_q|_{\CU_p \cap \CU_q} \]
that is compatible with the projections to \( \CG \). Furthermore, on any triple intersection, the uniqueness of this isomorphism implies that a cocycle condition holds. In other words, for all \( p,q,r \in M \) such that \(  \CU_p \cap \CU_q \cap \CU_r \neq \emptyset \) we have:
\[  [\CF_{qr}] \circ [\CF_{pq}] = [\CF_{pr}] \]
Since local groupoids are a sheaf. We conclude that there must exist a local groupoid \( \til \CG \) equipped with isomorphisms:
\[ [\phi_p] \colon \til \CG|_{\CU_p} \to \til \CG_p \]
which are compatible with the transition maps in the standard way.

Furthermore, we remark that for all \( p \in \CU \) we have a local groupoid chart:
\[  [\pi_p] \circ [\til \phi_p] \colon \til \CG|_{\CU_p} \to \CG \]
Compatibility of the collection \( \{[\til \phi_p ] \}_{p \in M} \) with the transition maps imply that the local groupoid charts \(\{ [\pi_p] \circ [\til \phi_p] \}\) are compatible on double intersections. Again, since morphisms of local groupoids are a sheaf, there must exist a morphism of local groupoids \( [\pi] \colon \til \CG \to \CG\). Since a local groupoid is canonically isomorphic to any neighborhood of its units, we can assume that \( \pi \) is globally defined. Furthermore such a \( \pi \) will be locally quasi-étale since each of the \( [\pi_p] \) are quasi-étale. Since the property of being quasi-étale is local, we conclude that \( \pi \) is quasi-étale and so \( \pi \colon \til \CG \to \CG \) is a local groupoid chart.

\end{proof}
\subsection{Proving Theorem~\ref{theorem:existence.and.uniqueness.of.lifts}}
Before we prove Theorem~\ref{theorem:existence.and.uniqueness.of.lifts}. Let us state new version of the equality test before that has been upgraded to account for maps between groupoids.
\begin{lemma}\label{lemma:relative.equality.test}
   Suppose \( \CG \grpd M \) and \( \CH \grpd N\) are local singular Lie groupoids and \( \pi_\CG \colon \til \CG \to \CG \) and \( \pi_\CH \colon \til \CH \to \CH \) are local groupoid charts.

    Suppose \( \CU \subset \til \CG \) is an open neighborhood of the image of \( \til \u \). Given a natural number \( n \), suppose that we have a pair of smooth functions:
    \[ \alpha \colon \CU^{(n)} \to \til \CG \qquad \beta \colon \CU^{(n)} \to \til \CG\]
    together with a smooth function:
    \[ f \colon M \to N \]
    with the following properties:
    \begin{itemize}
        \item \( \pi_\CH \circ \alpha = \pi_\CH \circ \beta \)
        \item  \( \til t \circ \alpha = f \circ \til t \circ \pr_1\)
        \item For all \( x \in \til u\inv(\CU) \), we have that  
        \[ \alpha(\til u(x), \ldots , \til u(x)) = \beta(\til u(x), \ldots , \til u(x)) = \til u(x) \]
    \end{itemize}
    Then there exists an open neighborhood \( \CO \subset \CU \) of the image of \( \til \u\) such that \( \alpha|_{\CO^{(n)}} = \beta|_{\CO^{(n)}} \)
\end{lemma}
The proof of this lemma is almost identical to the one for Lemma~\ref{lemma:equality.test}. In principal it involves proving a similarly reformulated version of Lemma~\ref{lemma:failure.function} as well. For the sake of avoiding repetition we will not write out the proof here except to remark that the only difference is addition of the function \( f \) in a few of the equations. One should also observe that the local groupoid structure on \( \til \CH \) will be a division structure so a version of Lemma~\ref{lemma:division.lemma} applies.

We can now proceed with proving Theorem~\ref{theorem:existence.and.uniqueness.of.lifts}
\begin{proof}
Suppose \( \CG \grpd M \) and \( \CH \grpd N \) are local singular Lie groupoids. Suppose \( [\CF] \colon \CG \to \CH \) is a local groupoid homomorphism covering \( f \colon M \to N \) and \( \pi_\CG \colon \til \CG \to \CG \) and \( \pi_\CH \colon \til \CH \to \CH \) are wide local groupoid charts.

We first show the uniqueness part. Suppose \( [ \til \CF ] \colon \til \CG \to \til \CH \) and \( [\til \CF' ] \colon \til \CG \to \til \CH \) are local groupoid homomorphisms with the properties:
\[ [\pi_\CH] \circ [\til \CF] = [\CF ] \circ [\pi_\CG] \qquad  [\pi_\CH] \circ [\til \CF'] = [\CF ] \circ [\pi_\CG] \]
We can assume that \( \til \CF \) and \( \til \CF' \) are globally defined in \( \til \CG\).

We invoke Lemma~\ref{lemma:relative.equality.test}, where we take \( \CU \) to be a common domain of definition for \( \til \CF \) and \( \til \CF'\), \( n = 1\), \( \alpha = \til \CF\) and \( \beta = \til \CF'\). It is straightforward to check that the hypotheses are satisfied and so we conclude there must exist an open neighborhood \( \CO \subset \CU \) of the units where \( \CF|_\CO = \CF'|_{\CO}\). Since local groupoid morphisms are germs in a neighborhood of the units we conclude \( [\CF] = [\CF']\).

Now we show existence. Note that since local groupoid morphisms can be defined locally in \( M \) (i.e. they are a sheaf), due to the uniqueness property we have just proved, it suffices to prove that there exists a \( \CF \) with the desired property that is defined in a neighborhood of an arbitrary point of \( M \).

Let \( x_0 \in M \) be a fixed, arbitrary point. Since \( \pi_\CH \colon \til \CH \to \CH\) is a local subduction, there must exist a smooth function \( \overline \CF \colon \CU \to \CH \) defined an open neighborhood \( \CU \subset \til \CG\) of \( \til u(x_0) \) with the property that 
\[\pi_\CH \circ \til \CF = \CF \circ \pi_{\CG} \] 
and 
\[\overline{\CF}(\til u(x_0)) = \til \u(f(x_0))  \]
Now let:
\[ \til \CF(g) := \overline{\CF}(g) \cdot \overline{\CF}(\til u \circ \til \s(g)) \inv\]
This may not be defined for all \( g \) but note that \( \til \CF(\til \u(x_0)) = \til \u(f(x_0))\). Therefore, there exists an open neighborhood of \( \til \u(x_0) \) where \( \til \CF \) is well-defined. Since we are only trying to show existence in a neighborhood of \( x_0 \) we can assume without loss of generality that \( \til \CF \) is defined on \( \CU \).

Now observe that \( \til \CF \) maps units to units. That is, for all \( x \in M \) such that \( \til \u(x) \in \CU \):
\[ \til \CF ( \til \u(x)) = \overline{\CF}(\til \u(x)) \cdot \overline{\CF}(\til \u(x))\inv = \til \u( \til \t \circ \overline{\CF} (\til \u(x))) \]
We claim that \( \til \CF \) defines a local groupoid homomorphism. To see this let \( \alpha\) and \( \beta \) be defined as follows. For \( (g,h) \in \CU^{(2)} \):
\[ \alpha (g,h) = \til \CF( gh) \qquad \beta(g,h) = \til \CF(g) \cdot \til \CF(h) \]
Again, we may need to shrink \( \CU \) to a smaller open neighborhood so that \( \alpha \) and \( \beta \) are well-defined but this is no issue.

Now we can invoke Lemma~\ref{lemma:relative.equality.test}. A straightforward check shows that \( \alpha \) and \( \beta \) satisfy the conditions of the equality test so we know there must exist an open neighborhood \( \CO \) of \( \til u(x_0) \) with the property that \( \alpha = \beta \) on this open set. Therefore, \( \CO \) is an open set where \( \til \CF \) is compatible with multiplication.

Finally, we need to show that \( \til \CF \) satisfies:
\[ [\pi_\CH] \circ [\til \CF] = [\CF ] \circ [\pi_\CG ] \]
However, if we compute for \( g \in \CO \):
\[ \pi_\CH \circ \til \CF(g) = \pi_\CH \left( \overline{\CF}(g) \cdot \overline{\CF}(\til \u \circ \s(g))\right) = \pi_\CH(\overline{\CF}(g))
\]
Where the last equality holds since \( \pi_{\CH} \) is a homomorphism. However, by definition of how we constructed \( \overline{\CF} \) we know that \( \pi_\CH \circ \overline{\CF} = \CF \circ \pi_\CG \). Therefore:
\[ [\pi_\CH] \circ [\til \CF] = [\CF ] \circ [\pi_\CG ] \]
This concludes showing local existence and so the proof is finished.
\end{proof}
\section{The classical Ševera-Weinstein groupoid}\label{section:Weinstein.groupoid}
We have now constructed the differentiation functor. Our final task is to prove that the classical Ševera-Weinstein groupoid is, indeed, a singular Lie groupoid.
\subsection{The fundamental groupoid construction}\label{section:weinstein.groupoid.construction}
    The Ševera-Weinstein groupoid is constructed via an analogy to the fundamental groupoid. Essentially, what one does is reproduce the construction of the fundamental groupoid but within the category of Lie algebroids.
    We will mostly follow the notation and approach to the construction that appears in the works of Crainic and Fernandes~\cite{crainic_integrability_2003}\cite{crainic_lectures_2011} as we will need to use several of their results about the geometry of the path space.
\begin{definition}
Suppose \( A \to M \) and \( B \to N \) are Lie algebroids and \( F_0 , F_1 \colon A \to B \) are Lie algebroid homomorphisms covering \( f_0, f_1 \colon M \to N \). We say that \( F_0 \) and \( F_1 \) are \emph{algebroid homotopic} if there exists a homomorphism of Lie algebroids \( H \colon A \times T [0,1] \to B \) covering \( h \colon M \times  [0,1]  \to N \) with the following properties:
\begin{enumerate}[(a)]
\item \( H|_{A \times 0_{\{ i \}}} = F_i \)
\item \( H|_{ 0_{\partial M} \times T[0,1] } = 0 \)
\end{enumerate}
\end{definition}
By concatenating homotopies, algebroid homotopies can be seen to form an equivalence relation on the set of all algebroid morphisms \( A \to B \). 
In particular, in the case of the tangent algebroid, the algebroid homotopy relation coincides with the standard notion of homotopy of smooth maps.
\begin{definition}
Suppose \( A \to M \) is a Lie algebroid. An \( A \)-path is an algebroid morphism \( a \colon T [0,1] \to A \). The set of all \( A \)-paths will be denoted \( \CP(A) \). 
The \emph{fundamental groupoid} or \emph{Ševera-Weinstein groupoid} of \( A \) is denoted \( \Pi_1(A) \) and is defined to be the set of all \( A \)-paths modulo algebroid homotopy.
\end{definition}
In the work of Crainic and Fernandes, they make use of the theory of Banach manifolds to study the structure of \( \CP(A) \). 
In order to do this, we must relax the smoothness assumptions on \( \CP(A) \). 
From now on, we assume that any \( a \in \CP(A) \) is a bundle map of class \( C^1 \) and has base path of class \( C^2 \). 
Under these assumptions, Crainic and Fernandes prove the following theorem:
\begin{theorem}\cite{crainic_integrability_2003}
The set of algebroid homotopy equivalence classes in \( \CP(A) \) partitions \( \CP(A) \) into a foliation with finite co-dimension \( \CF \).
\end{theorem}
In addition to this, Crainic and Fernandes also show how this construction can be used to obtain a local Lie groupoid.
\begin{theorem}\cite{crainic_integrability_2003}
For any Lie algebroid \( A \), there exists an open neighborhood \( A^\circ \) of the zero section together with an embedding \( \exp \colon A^\circ \to \CP(A) \) with the following properties: 
\begin{itemize}
\item \( \exp \colon A^\circ \to \CP(A) \) is transverse to the foliation \( \CF \).
\item There is a local Lie groupoid structure on \( A^\circ \) which makes the associated map \( \pi \colon A^\circ \to \Pi_1(A) \) into a local groupoid homomorphism.
\item The Lie algebroid of \( A^\circ \) as a local groupoid is canonically isomorphic to \( A \).
\end{itemize}
\end{theorem}
An important consequence of the first bullet point above is that \( \exp \colon A^\circ \to \Pi_1(A)  \) is a local subduction.
This is essentially implicit in the literature, however it is a little difficult to find a clear statement. This is due to the fact that it relies of properties of foliations on Banach manifolds.

For example, in \cite{tseng_integrating_2006}, it is claimed that the restriction of the fundamental groupoid of the foliation \( \CF \) to an arbitrary complete transversal results in an étale Lie groupoid. Therefore, by Lemma~\ref{lemma:lie.groupoid.quotient.is.subduction} passing to the quotient space should be a local subduction. Unfortunately, Tseng and Zhu do not go into much detail beyond stating this fact.

For completeness, we will include a proof of this fact in the following lemma.
\begin{lemma}\label{lemma:exponential.is.subduction}
Let \( \pi \colon A^\circ \to \Pi_1(A) \) be the map defined above. Then \( \pi \) is a local subduction.
\end{lemma}
\begin{proof}
    First, we remark that since \( \CF \) is a foliation on a Banach manifold, around any point \( a(t) \in \CP(A)\) there exists a foliated chart \(\Phi \colon  \RR^k \times V \to \CP(A) \). 
    By foliated chart, we mean that \( \Phi \) is an open diffeomorphism onto its image, \( k \) is the codimension of \( \CF\), \( V \) is an open subset of a Banach vector space, and the foliation \( \CF \) corresponds to \( \RR^k \times TV \).

    The existence of such foliated charts for foliations on Banach manifolds seems a little difficult to find in the literature stated plainly in this way. 
    However, it can be inferred from the proof of Theorem~1.1 in Chapter 7 of Lang~\cite{lang_fundamentals_1999}.

    Now, to see why \( \exp \) is a subduction suppose \( \phi \colon U_\phi \to \Pi_1(A) \) is a plot and let \( v \in A^\circ \) and \( u \in U_\phi \) be such that \( exp(v) = \phi(u) \).

    Let \( p \colon \CP(A) \to \Pi_1(A) \) be the canonical projection so that \( \pi = p \circ \exp \). 
    Since \( \Pi_1(A) \) is equipped with the quotient diffeology, we know there exists a lift \( \til \phi \colon V \to \CP(A) \) defined on some open neighborhood \( V\) of \( u\).
    
    Now suppose we have a submanifold \( \CT \subset \CP(A) \) which is transverse to \( \CF \), and is such that the image of \( \phi \) is contained in \( \CT\).
    By using a foliated charts centered around \( \phi(u) \), and \( a \in \CP(A)\) there exists a locally defined function \( q_1 \colon \CO_1 \to \CT \) and \( q_2 \circ \CO_2 \to A^\circ \) defined around an open neighborhood \( \CO_1 \) of \( \phi(u) \) and \( \CO_2 \) of \( \exp(v) \) such that \( p \circ q_1 = p\) and \( p \circ q_2 = p\).
        
    Now, we also know that \( \til \phi(u) \) is connected to \( a \) by a path \( h(t)\) in \( \CP(A) \) which is tangent to \( \CF \).
    Since \( \CT \) is a transverse submanifold at \( \phi(u)\) and \( \exp(A^\circ)\) is a transverse submanifold at \( \exp(v) \) we can use the classical construction\footnote{Note that the traditional form of this construction is by linking a series of foliated charts.} for the holonomy of a foliation, to obtain a germ of a diffeomorphism
    \[ f \colon \CT \to \exp(A^\circ) \]
    such that \( f(\phi(u)) = \exp(v) \) and \( p \circ f = p\).

    Therefore, by choosing a suitably small open neighborhood \( V \) of \( u \) the function 
    \[q_2 \circ f \circ q \circ \til \phi  \] 
    is a lift of \( \phi \) with the property that \( \phi(u) = a\).
    This shows that \( \pi \) is a local subduction.
\end{proof}
In the article of Crainic and Fernandes, the main goal was to obtain a criteria for integrability.
Towards this end, they became interested in understanding the fibers of this exponential map and, in particular, they computed the inverse images of the unit elements.
\begin{theorem}[Crainic and Fernandes]
    Let \( A \) be a Lie algebroid and let \( \exp \colon A^\circ \to \Pi_1(A) \) be as in the previous theorem.
    
    Given \( x \in M \) the fiber \( \exp\inv( \u(x)) \) is countable.
\end{theorem}
Indeed the actual theorem proved in the article of Crainic and Fernandes is much stronger:
The fiber \( \exp\inv(\u(x)) \) is actually an additive subgroup of \(A \) called the \emph{monodromy group} at \( A \) and it arises from a group homomorphism from the second fundamental group of the orbit.
The countability of the monodromy group is just an immediate consequence of the countability of the fundamental groups of a smooth manifold.

\subsection{The Ševera-Weinstein groupoid is a singular Lie groupoid}
In this subsection we will explain the proof of Theorem~\ref{theorem:main.wein.groupoid}. Let us restate this theorem now.

\begin{reptheorem}{theorem:main.wein.groupoid}
    Given a Lie algebroid, \( A \to M \), let \( \Pi_1(A) \grpd M  \) be the Ševera-Weinstein groupid of \( A \) and consider it as a diffeological groupoid. Then \( \Pi_1(A) \) is an element of \( \SingLieGrpd \) and \( \hat\Lie(\Pi_1(A)) \) is canonically isomorphic to \( A \).
\end{reptheorem}

Our first lemma is used to argue that we only have to check if \( \Pi_1(A) \) is a singular Lie groupoid in a neighborhood of the units.
\begin{lemma}\label{lemma:locally.Weinstein.is.weinstein}
Suppose \( \CG \grpd M \) is a diffeological groupoid over a smooth manifold \( M \). Assume that the source map \( s \colon \CG \to M \) is a local subduction. Now suppose there exists an open neighborhood of the units \( \CU \subset \CG  \) such that \( \CU \) is a singular local Lie groupoid. Then \( \CG \) is a singular Lie groupoid. 
\end{lemma}

\begin{proof}
    The proof of this fact is a translation argument.
    
    We need to show that \( \CG \) is a \( \QUED\)-groupoid. This means that we must show that \( \CG \) is a quasi-étale diffeological space and the source map \( s \colon \CG \to M \) is a \( \QUED\)-submersion.

    For the first part, let \( g_0 \in \CG \) be arbitrary and let \( s(g) = x_0 \in M \). Since \( \CU \) is quasi-étale, there must exist a quasi-étale chart \( \pi \colon N \to \CG \) where \( \u(x) \) is in the image of \( \pi\). Let \( n_0 \in N \) be the point such that \( \pi(n_0) = x_0 \). 
    
    Now let \( \sigma \colon \CO \to \CG \) be a local section of the source map such that \( \sigma(x) = g \). Then we can define:
    \[ \pi'(n) := \sigma( t \circ \pi(n)) \cdot \pi(n) \]
    Clearly \( \pi' \) will be well-defined in an open neighborhood of \( n_0\). Furthermore, \( \pi'(n_0) = g_0 \). Since \( \pi' \) is locally \( \pi \) composed with a diffeomorphism, it follows that \( \pi' \) is quasi-étale. Therefore, \( \pi' \) is a quasi-étale chart around \( g_0 \).

    To show that \( s \colon \CG \to M \) is a \( \QUED\)-submersion, we will utilize Theorem~\ref{theorem:qeds.submersion.criteria}. We only need to show that the fibers of \( s\) are quasi-étale diffeological spaces. However, we already know that \( \CU \) is a singular local Lie groupoid. Therefore, for each \( x \in M \) we know \( s\inv(x) \cap \CU \) is a quasi-étale diffeological space. To show that \( s\inv(x) \) is a quasi-étale diffeological space we once again invoke a simple translation argument. Any point in \( s\inv(x) \) is just a left translation of the identity. Therefore, a quasi-étale chart around the identity element of \( s\inv(x) \) can be translated to be a quasi-étale chart around any element of \( s\inv(x)\).
\end{proof}
In order to show that \( \Pi_1(A) \) is a singular Lie groupoid we will have to show that the source map is a \( \QUED\)-submersion. Part of that is the claim that the source map is a local subduction. In a way, this is also a bit implicit in the literature but we include a proof here.
\begin{lemma}\label{lemma:source.is.subduction}
    Suppose \( p \colon A \to M \) is a Lie algebroid. Let \( s \colon \Pi_1(A) \to M \) be the source map for the fundamental groupoid. Then \( s \) is a local subduction.
\end{lemma}
\begin{proof}
       First, we argue that the map:
       \[ \hat s \colon \CP(A) \to M \qquad a \mapsto p(a(0)) \]
       is a submersion of Banach manifolds.

       To see this, we rely on the original paper of Crainic and Fernandes~\cite{crainic_integrability_2003} where they compute \( T_a \CP(A) \) to be:
       \[ \{ (u,X) \in C^\infty([0,1], A \times TM) \ : \ \overline{\nabla}_a X  = \rho(u)  \} \]
       where \( \nabla\) is a connection on \( A \), \( \rho\) is the anchor map and:
       \[ \overline{\nabla}_a X := \rho( \nabla_X a) + [\rho(a),X ] \]
       The right hand side of this expression is computed by using any set of sections that extend \( X \) and \( a\) to time dependent sections of \( TM \) and \( A \), respectively.

       Actually, \( \overline{\nabla} \) is an example of what is called an \( A \)-connection. 
       More specifically, \( \overline{\nabla} \) is an \( A \) connection on the vector bundle \( TM \). 
       The expression \( \overline{\nabla}_a \phi \) should be interpreted as the directional derivative of \( \phi \) along \( a \). 
       We refer the reader to \cite{crainic_lectures_2011} for more details on the geometry of \( A \)-connections. 
       
       The key fact is that, for \( A \) connections, parallel transport along an \(A \)-path is well-defined and 
       the equation \( \overline{\nabla}_a X = 0 \) is equivalent to stating that \( X \) is parallel along \( a\).

       Now, the \( u \) part represents the vertical (fiber-wise) component of the tangent vector while \( X \) represents the ``horizontal'' component of the tangent vector. In particular, if we consider the differential of the projection we get:
       \[ T_a p : T_a \CP(A) \to T_{p \circ a} C^2([0,1],M)  \qquad (u,X) \mapsto X \]
       Therefore we have:
       \[ T_a \hat s \colon T_a \CP(A) \to T_{p \circ a(0)} M \qquad (u,\phi) \mapsto \phi(0) \]
       If we are given any \(v \in  T_{p \circ a(0)} M \) and we let \( \phi \) be the unique solution to the initial value problem:
       \[ \overline{\nabla}_a \phi = 0 \qquad \phi(0) = v \]
       Then it follows that \( (0,\phi) \in T_a \CP(A)\) is a tangent vector with the property that \( T_a \hat s ( 0,\phi) = v \). This shows that \( \hat s \) is a submersion of Banach manifolds.

       Now, since the co-domain of \( \hat s \) is finite dimensional, the kernel of \( T \hat s \) must admit a closed complement at each point. Therefore, the inverse function theorem for Banach manifolds applies and \( \hat s \) admits local sections through every point in its domain. This immediately implies that \( \hat s \) is a local subduction. As an immediate consequence, \( s \colon \Pi_1(A) \to M \) is also a local subduction.
\end{proof}
In a neighborhood of the identity, the fundamental groupoid of an algebroid is just a quotient of \( A^\circ \). Therefore, it will be useful to know, precisely, which kinds of quotients of local groupoids are singular Lie groupoids. The next theorem answers this question.
\begin{theorem}\label{theorem:classify.qeds.groupoids}
Suppose \( \CG \grpd M\) is a diffeological groupoid where the source map is a local subduction. Suppose further that \( \til \CG\) is a local Lie groupoid over the same manifold and we have \( \pi \colon \til \CG \to \CG \), a local subduction, such that \( [\pi] \colon \til \CG \to \CG\) is a homomorphism of local groupoids covering the identity.

If for all unit elements \( u \in \CG \) we have that the fiber \( \pi\inv(u) \subset \til \CG \) is totally disconnected, then \( \CG \grpd M\) is a \( \QUED\)-groupoid and \( \pi \) is a wide local groupoid chart.
\end{theorem}
\begin{proof}
    By Lemma~\ref{lemma:locally.Weinstein.is.weinstein} it suffices to show that the image of an open neighborhood of the units in \( \til \CG \) under \( \pi \) is a local singular Lie groupoid.

    To set our notation, we will use \( \til s, \til t , \til \u , \til \m \) and \( \til \i \) to denote the local groupoid structure on \( \til \CG\) and \( s,t, u, m, i\) the local groupoid structure on \( \CG \).

    We begin by arguing that \( \pi \) is a quasi-étale chart (in some neighborhood of the units). We already know that \( \pi \) is a local subduction.
    Therefore, we need to show that \( \pi \) has totally disconnected fibers (in a neighborhood of the units) and satisfies the rigid endomorphism property. 

    We know that \( \pi \) has totally disconnected fibers over the units of \( \CG \). We claim that the fibers of \( \pi \) are totally disconnected in a neighborhood of the units of \( \til \CG \).

    To see why, consider the following calculation:
    \[ (b \cdot  a\inv) \cdot a = b \cdot (a\inv \cdot a) = b \cdot (\u(\s(a))) = b \]
    Let \( \CU \subset \til \CG \) be an open neighborhood of the identity with the property that for all \( a,b \in \CU \) with \( \s(a) = \s(b)\) we have that the above calculation is well-defined in \( \til \CG \).

    Now given \( \til g_0 \in \CU \) and let \( g_0 = \pi(\til g)\). Consider the function:
    \[ d \colon \pi\inv(g_0) \cap \CU \to \til \CG \qquad x \mapsto x \cdot {\til g_0}\inv  \]
    This function is well-defined since we are in \( \CU \). Furthermore, this function is injective due to the fact that if we have \( x \cdot a \inv = y \cdot a \inv \) we have that:
    \[ x = (x \cdot  a \inv ) \cdot a = (y \cdot a\inv ) \cdot a = y  \]
    Furthermore, \( d_a \) takes values in the kernel of \( \pi \). Specifically, it takes values in \( \pi\inv(\u(t(g_0)))\), which are totally disconnected by assumption. Since \( d\) is a smooth injection into a totally disconnected set, it follows that the domain of \( d \) is totally disconnected.

    This shows the fibers of \( \pi \) are totally disconnected when \( \pi \) is restricted to a small enough neighborhood of the units. Therefore, we can assume without loss of generality that, from now on, the fibres of \( \pi \) are totally disconnected.

    Next, we must show that \( \pi \) satisfies the rigid endomorphism property.

    Let us take \( \CU \subset \til \CG \) to be an open subset of \( \til \CG \) with the property that for all \( x , y \in \CU \) such that \( s(x) = s(y) \) we have that \( x \cdot y\inv \) is well-defined.

    We use the simplified criteria from Lemma~\ref{lemma:finding.quasi.étale.charts} to show the rigid endomorphism property.
    Let \(  g_0 \in \CU \) be arbitrary and \( f \colon \CO \to \til \CG\) be a smooth function with the property that \( \CO \) is an open neighborhood of \(  g_0 \), \( f(  g_0 ) =  g_0 \), and \( \pi \circ f = \pi \).
    We only need to show that \( f \) is a diffeomorphism in a neighborhood of \(  g_0 \).

    Let \( \alpha\) be the following function:
    \[ \alpha \colon \CO \to \til \CG \qquad f(g)\cdot g\inv \]
    Note \( \alpha \) will be well-defined by our assumptions on \( \CU \).

    If we apply \( \pi \) to \( \alpha \) we observe that \( \pi \circ \alpha \) takes values only in units. Now let us chose \( \CO \) to be an open neighborhood of \( g_0\) with the property that \( \CO \cap \til t\inv(x) \) is connected for all \( x \in M\). Since the fibers of \( \pi \) are totally disconnected, it follows that \( \alpha \) must be constant on the \( \til t\)-fibers. 

    In particular, there must exist a local section \( \sigma \colon \til t (\CO) \to \til \CG \) with the property that:
    \[ \forall g \in \CO \qquad \alpha(g) = \sigma(\til t(g))  \]
    From this we can conclude that:
    \[ \forall g \in \CO \qquad f(g) \cdot g\inv = \sigma(\til t(g)) \]
    Which means that:
    \[ \forall g \in \CO \qquad f(g) = \sigma(\til t(g)) \cdot g \]
    From this it immediately follows that \( f \) is a diffeomorphism since it is given by translation along a section. In fact, it has an explicit inverse:
    \[ f\inv(g) = \sigma(\til t(g))\inv \cdot g \]

    To conclude showing that the image of \( \pi \) is a local singular Lie groupoid, the only thing left to show is that the source map \( s \colon \CG \to M \) is a \( \QUED\)-submersion. We will use the criteria from Theorem~\ref{theorem:qeds.submersion.criteria} which means that we need to show that for all \( x \in M \) we have that \( s\inv(x) \) is a quasi-étale diffeological space.

    Therefore, we just need to show that \( \pi|_{\til \s\inv(x)} \) is a quasi-étale chart. 
    Let \( X := \pi(\til \s\inv(x)) \). 
    We already know it is a local subduction and that the fibers are totally disconnected. 
    It only remains to show that it satisfies the rigid endomorphism property.

    We use the simplified criteria from Lemma~\ref{lemma:finding.quasi.étale.charts}. Suppose \( \CO \subset \til s\inv(x) \) is open and \( f \colon \CO \to \til \s\inv(x)  \) is a smooth function such that \( \pi \circ f = \pi \) and \( f(\til g_0) = \til g_0 \) for some point \(\til g_0 \in \CO \). We need to show that \( f \) is a diffeomorphism in a neighborhood of \( \til g_0 \).
  
    We need to show that \( f \) is a local diffeomorphism. The argument is very similar from the one earlier in the proof. Let:
    \[  \alpha \colon \CO \to \til \s\inv(x) \qquad \alpha(g) := f(g) \cdot g\inv\]
    Applying \( \pi \) to \( \alpha \) we get that:
    \[ \pi \circ \alpha  (g) = \u \circ t(g) \] 
    If we chose \( \CO \) such that the intersection with the \( t \)-fibers is connected, then we can conclude that \( \alpha \) is constant when restricted to \( t \)-fibers.
    
    Therefore, there must exist a smooth section of the source map \( \sigma \colon \t(\CO) \to \til \CG  \) with the property that:
    \[ \alpha(g) = \sigma( \til t(g)) \]
    Therefore, we can conclude that:
    \[ 
    f(g) \cdot g\inv = \sigma(\til t(g)) \] 
    Therefore:
    \[f(g) = \sigma(\til t(g)) \cdot g \]
    Which, by similar reasoning, is a diffeomorphism.
\end{proof}
We can now finish the proof of the main theorem for this section.
\begin{proof}[Proof of Theorem~\ref{theorem:main.wein.groupoid}]
        Let \( \pi \colon A^\circ \to \Pi_1(A) \) be the projection from Section~\ref{section:weinstein.groupoid.construction}. Recall that \( A^\circ \) is an open subset of the zero section in \( A \) and that \( A^\circ \) is equipped with a unique local groupoid structure that makes \( [\pi] \colon A^\circ \to \Pi_1(A) \) a local groupoid homomorphism and the zero section of \( A^\circ \) the unit embedding.

        Note that by Lemma~\ref{lemma:source.is.subduction} we know the source map of \( \Pi_1(A) \) is a local subduction. Furthermore, we know that \( \pi \) is a subduction by Lemma~\ref{lemma:exponential.is.subduction}. By Theorem~\ref{theorem:classify.qeds.groupoids} we conclude that \( \Pi_1(\CG) \) is a singular Lie groupoid.

        Furthermore, the map \( \pi \colon A^\circ \to \Pi_1(\CG) \) is a local groupoid chart. Therefore, the Lie algebroid of \( \Pi_1(\CG)\) is canonically isomorphic to the Lie algebroid of the local groupoid structure on \( A^\circ\) which is just \( A \).
    \end{proof}

\bibliographystyle{halpha}
\bibliography{references.bib}

\end{document}